%% file: main.tex
\documentclass[10pt, a4paper, twoside, leqno]{article}
    \makeatletter
    \renewcommand*{\@fnsymbol}[1]{\ensuremath{\ifcase#1\or 1\or 2\or \ddagger\or
        \mathsection\or \mathparagraph\or \|\or **\or \dagger\dagger
        \or \ddagger\ddagger \else\@ctrerr\fi}}
    \makeatother
    
\usepackage{paper}

\title{\Large\textsc{\bfseries{Markov-Modulated Affine Processes}}} 
\date{\vspace{0em}\normalsize \today}
\author{\textsc{Kevin Kurt}\thanks{Corresponding author, Vienna University of Economics and Business (WU Wien), Institute for Statistics and Mathematics, {\tt kevin.kurt@wu.ac.at}} \and
\textsc{R\"udiger Frey}\thanks{Vienna University of Economics and Business (WU Wien), Institute for Statistics and Mathematics, {\tt ruediger.frey@wu.ac.at}}}

\begin{document}

\maketitle

\begin{abstract}
\small{
We study Markov-modulated affine processes (abbreviated MMAPs), a class of Markov processes that are created from affine processes by allowing some of their coefficients to be a function of an exogenous Markov process $X$.
MMAPs largely preserve the tractability of standard affine processes, as their characteristic function has a computationally convenient functional form. 
Our setup is a substantial generalization of earlier work, since we consider the case where the generator of $X$ is an unbounded operator.
We prove existence of MMAPs via a martingale problem approach, we derive the formula for their characteristic function and we study various mathematical properties.
} 
\end{abstract}
\small{
\begin{paragraph}{\small{Keywords}}
Markov processes $\cdot$ Affine processes $\cdot$ Martingale problem $\cdot$ Analytical tractability $\cdot$ Pricing of financial instruments $\cdot$ Markov processes with discontinuous coefficients
\end{paragraph}
\vspace*{-1em}
\begin{paragraph}{\small{Mathematics Subject Classification (2020)}}
60J25 $\cdot$ 60J35 $\cdot$ 91B70
\end{paragraph}
\vspace*{-1em}
\begin{paragraph}{\small{JEL Classification}}
C02 $\cdot$ G12
\end{paragraph}
}



\normalsize
\pagestyle{fancy}
\input{1}

\input{2}
\input{3}

\input{4}

\input{5}
\input{6}

\appendix
\input{appendix_defs}
\input{appendix_proofs}

\section*{Acknowledgments}
We are grateful to an anonymous referee and to the Associate Editor for their helpful comments that have led to a substantial improvement of our results and the present article.
We also like to thank Christa Cuchiero and Thorsten Schmidt for their valuable suggestions.

\label{References}
\bibliography{references}
\vspace{1.5cm}

\noindent \textbf{Affiliations:}

\vspace{.5cm}

  \noindent Kevin Kurt\\
  Institute for Statistics and Mathematics\\
  Department of Finance, Accounting and Statistics\\
  WU Vienna University of Economics and Business\\
  Welthandelsplatz~1 / Building D4 / Level 4\\
  1020 Vienna, Austria\\
  E-mail: \texttt{kevin.kurt@wu.ac.at}
\\\\
  R\"udiger Frey\\
  Institute for Statistics and Mathematics\\
  Department of Finance, Accounting and Statistics\\
  WU Vienna University of Economics and Business\\
  Welthandelsplatz~1 / Building D4 / Level 4\\
  1020 Vienna, Austria\\
  E-mail: \texttt{ruediger.frey@wu.ac.at}

\end{document}

%% file: 1.tex
\section{Introduction}

A standard affine process $Y$ in the sense of~\cite{Duffie2003AffineFinance} can be characterized as a strong Markov process whose generator is given in terms of so-called \emph{admissible} parameters that form affine functions of the state $y$ of the process, cf.\ \cite{Keller-Ressel2011AffineRegular}.
{Markov-modulated affine processes} or MMAPs are a natural extension where the constant (non $y$-dependent)  part of these affine functions  depends on  some exogenous Markov process $X$.
A  simple example is provided by a  CIR process with Markov modulated mean reversion level.
Formally, we introduce Markov-modulated affine processes as a class of Markov processes, which we define in terms of the generator $\gen$: we assume that $\gen$ is a linear operator of the form
\[
\gen f(x, y) = \genX f(x, y) + \genY f(x, y), \quad f\in \domgen{\gen} \subset \Cinf{\stateX \times \stateY}.
\]
Here, $\genX$ (acting on $ x \mapsto f(x,y)$) is the generator of a Feller process $X$ with state space $\stateX \subset \R^d$, and for fixed $x$ the operator $\genY$ (acting on $ y \mapsto f(x,y)$) is the generator of an affine process $Y$  with state space $\stateY \subset \R^n$.
Thus, conditional on the path of $X$, the process $Y$ can be regarded as a time-inhomogeneous affine process.

Markov-modulation adds  flexibility for (financial) modelling to the class of  affine processes.
At the same time MMAPs largely  preserve the tractability of the latter class, as the  characteristic function of their  marginal distributions  can be computed via an  extension of the Riccati equations for standard affine processes.
More precisely, for an MMAP it holds that
\begin{equation}
    \label{eqn:charfun_ch1}
\EVM{z}{}{e^{\innerp{u}{Y_t}}} = \varphi(t, x; u) e^{\innerp{\psi(t, u)}{y}}, \quad z= (x,y) \in \stateX \times \stateY.
\end{equation}
The function $\psi$ is characterized by the same  system of generalized Riccati equations as in the standard affine case, whereas $\varphi$ is closely connected to a Cauchy problem involving $\psi$ and the generator $\genX$ of $X$.

Simple \emph{regime-switching} MMAPs where $X$ is a finite state Markov chain have been used previously in  finance.
For instance, \cite{Elliott2009OnFormulae} consider bond pricing in affine short rate models with regime switching; \cite{Frey2020HowModelsb}  use a portfolio credit risk model with default intensities given by a CIR process with regime switching to analyze securitization products backed by European sovereign bonds, and  \cite{Elliott2006OptionCompensatorsb}  study option pricing for regime switching  pure jump Lévy processes.
The recent contribution \cite{vanBeek2020RegimeFinance} unifies the previous examples and  provides a theory of  MMAPs for the case where the generator $\genX$ is a \emph{bounded} operator so that $X$ is a finite-activity jump process.

The present paper is a substantial extension of \cite{vanBeek2020RegimeFinance}:   we make only very mild assumptions on $\genX$  and  we allow for discontinuities in the  coefficients of $\genY$. 
In particular,   $\genX$ may be unbounded, so that  models where $X$ is a (jump-)diffusion or a jump process with infinite activity fall  within the scope of our analysis. 
This is relevant for financial applications.
There $X$ often represents state variables such as market sentiment or the economic environment, and it may be more natural to model the dynamics of these state variables by continuous processes and not by processes with piecewise constant trajectories; see also the financial applications discussed in Section~\ref{sec:applications}. 
The extension to unbounded generators  $\genX$ is also interesting from a mathematical viewpoint.
In that case the generator $\gen$ of $(X, Y)$ does not satisfy  the regularity conditions commonly imposed in the construction of Markov processes via   perturbation arguments, so that classical results from semigroup theory (cf.\ \cite[Section~3.3]{Pazy1983SemigroupsEquations} or \cite[Section~1.7]{bib:ethier-kurtz-86}) are not readily applicable. 
To deal with this issue  we choose a probabilistic approach and  use weak convergence results  to construct solutions to  the martingale problem associated with $\gen$.
From relation~\eqref{eqn:charfun_ch1} we can further determine the marginal distributions of $(X,Y)$ jointly.
Classical results for martingale problems thus guarantee  the Markov property and consequently the existence of MMAPs.
We go on and analyze further mathematical  properties of MMAPs such as the Feller property,  the existence of real exponential moments and the semimartingale characteristics of  MMAPs.
Finally, in order to illustrate the wide range of modelling possibilities offered by MMAPs we discuss several applications of MMAPs in finance.

Our analysis builds  on the formal treatment of affine processes provided in~\cite{Duffie2003AffineFinance}.
Moreover, we make extensive use of the comprehensive treatment of Markov processes and martingale problems in \cite{bib:ethier-kurtz-86}.
We further contribute to the list of extensions to affine processes that has started to grow ever since the seminal work of~\cite{Duffie2003AffineFinance}.
\cite{Filipovic2005Time-inhomogeneousProcesses} introduces time-inhomogeneous affine processes.
\cite{Cuchiero2011AffineMatrices} consider matrix-valued affine processes and~\cite{cuchiero2016affine} subsequently generalize the state space by considering affine processes on symmetric cones.
More recently, \cite{keller2019affine} study a setup where the jump times of affine processes are allowed to be predictable (so that the processes no longer obey stochastic continuity) and \cite{Schmidt2020InfiniteProcesses} discuss infinite dimensional affine diffusions.

The rest of the paper is organized as follows.
Section~\ref{sec:MMAPs} introduces the necessary technical preliminaries and the formal notion of MMAPs.
In particular, we discuss the specifics of the operator  supposed to act as the generator of our class of processes.
In Section~\ref{sec:MP} we solve the martingale problem associated with this operator.
Section~\ref{sec:fourier} is concerned with  the Fourier transformation of MMAPs and the associated existence result. Further mathematical properties of MMAPs are discussed in Section~\ref{sec:further-properties}.
Finally, in Section~\ref{sec:applications} we highlight the applicability of MMAPs by discussing novel models from mathematical finance.

%% file: 2.tex
\section{Setup}\label{sec:MMAPs}

In this section we introduce the necessary notation and we give a formal definition of a Markov modulated affine process (abbreviated MMAP) via its generator.

\subsection{Notation and basic concepts}\label{subsec:setup}

\paragraph*{Analysis} Let $E$ be either an open set or the closure of an open set in $\C^k$.
Then, $\Bb{E}$ is the space of bounded Borel measurable functions,  $C(E)$ ($\Cb{E}$) denotes the Banach space (equipped with supremum norm $\norm{\cdot}_\infty$) of continuous (bounded) functions on $E$.
The complete subspace
$$\Cinf{E} = \{f \in C(E) \, : \, \forall \epsilon > 0 \,\, \exists \, K \subset E \text{ compact}: |f| < \epsilon \text{ on } E\setminus K\}$$ is the space of continuous functions vanishing at infinity.
We use $\Ck{k}{E}$ to denote the space of $k$ times differentiable functions $f$ on the interior of $E$ such that all partial derivatives of $f$ up to order $k$ belong to $C(E)$, and where $\Ck{\infty}{E} = \cap_{k \in \N} \Ck{k}{E}$  and $\Schwartz{m} \subset \Ck{\infty}{\R^m}$ denotes the Schwartz space of rapidly decreasing functions on $\R^m$.
The space $\Ckb{k}{E}$ consists of bounded $\Ck{k}{E}$-functions (where the derivatives are bounded as well).
The space of functions in $\Ck{k}{E}$ with compact support is $\Ckc{k}{E}$.
The set of càdlàg functions on $[0, \infty)$ with values in some space $A$ is denoted by $D_A$.
When there is no ambiguity with respect to  dimensionality we write $\bm{1}$ for a vector of ones.
Similarly, we use $\Id$ to denote the identity.
We write $\bplim{r \to a} f_r$ to indicate bounded and pointwise convergence of $\{f_r\}_{r \in I} \subset \Cb{E}$ for some index set $I$.

In this paper, a MMAP is defined as a $(d + n)$-dimensional process $Z=(X,Y)$ on the state space $\smash{\stateXY = \stateX \times \stateY}$, where $\stateX$ is an open (or the closure of an open) set of $\R^d$ and $\smash{\stateY = \R^m_+ \times \R^{n-m}}$.
Throughout we use $x$ to denote a point in $\stateX$ and $y = (y^+, y^*)$ denotes a point in $\stateY$ with $y^+ \in \R^m_+$ and $y^* \in \R^{n-m}$.
An element in $\C^n = \C^m \times \C^{n-m}$ is denoted by $u = (u^+, u^*)$.
We follow the notational conventions established in \cite{Duffie2003AffineFinance} and introduce $\indexI := \{1, \dots, m\}$ and $\indexJ := \{m+1, \dots, n \}$ along with $\indexI(i) := \indexI \setminus \{i\}$ and $\indexJ(i) := \{i\} \cup \indexJ$.
For a generic $k\times k$-matrix $\alpha = (\alpha_{ij})$ and a $k$-tuple $\beta = (\beta_1, \dots, \beta_k)$, we write $\alpha_{IJ} := (\alpha_{ij})_{i \in I, j \in J}$ and $\beta_I := (\beta_i)_{i \in I}$ for indices $I, J \subset \{1, \dots, k\}$.
For $u \in \C^n$, we write $f_u(x) := e^{\innerp{u}{x}}$.
Throughout we use the following truncation function $\chi = (\chi_1, \dots, \chi_n) : \R^n \to [-1,1]^n$ given by 
\[
\chi_k(\xi) := \begin{cases} 0, \quad &\text{if } \xi_k = 0, \\ (1 \wedge |\xi_k|) \frac{\xi_k}{|\xi_k|}, &\text{otherwise.} \end{cases}
\]
The Borel $\sigma$ fields on  $\stateX$, $\stateY$ and $\stateXY$ are denoted by $\BorelSigmaAlg{\stateX}$, $\BorelSigmaAlg{\stateY}$ and $\BorelSigmaAlg{\stateXY}$.

\paragraph*{Probability} In order to avoid ambiguities with respect to different notions of  \emph{Feller} semigroups in the literature, we collect the definitions used in this paper in Appendix~\ref{appendix:defs}.
Since affine semigroups are in general \emph{not conservative}, we throughout  work on the path space ${\Omega} := D_{\stateXY^\Delta}$, where $\stateXY^\Delta := \stateXY \cup \{\Delta\}$ is the one-point compactification of $\stateXY$ with $\Delta \notin \stateXY$ describing the point at infinity.
We extend any function $f$ on $\stateXY$ to $\stateXY^\Delta$ via $f(\Delta) = 0$ and we endow $\Omega$ with the \emph{Skorohod} topology $\topology$.
Considering the product structure of the state space, we denote the canonical process by $Z_t(\omega) := (X_t(\omega), Y_t(\omega)) := \omega_t, \, t \geq 0$ for $\omega \in \Omega$.
Note that $(\Omega, \topology)$ is a Polish space and that the corresponding Borel $\sigma$-algebra is generated by the evaluation maps, see \cite[Section~VI.1b]{bib:jacod-shiryaev-03} for further details.
We use $(\calF_t^0)_{t \geq 0}$ to denote the filtration generated by $Z$ and we set $\calF^0 := \bigvee_{t \geq 0} \calF^0_t$.
Note that working on $D_{\stateXY^\Delta}$ is no restriction as \cite{cuchiero2013path} show that an affine process always has a càdlàg version.
We write $\probset{\Omega}$ for the set of probability measures on $\Omega$ and $\Pp_z \in \probset{\Omega}$ denotes a measure with $\Pp_z(Z_0 = z) = 1.$

\subsection{Markov-Modulated Affine Processes}

We aim to construct a family of probability measures $(\Pp_z)_{z\in \stateXY} \subset \probset{\Omega}$, for which $(Z, (\Pp_z))$ is a Markov process with the following properties:
$X$ modulates the coefficients of the generator of $Y$ and conditional on the paths of $X$ the process $Y$ is a time-inhomogeneous affine process.
For the characterization of $(X,Y)$ we rely on the martingale problem approach associated to a linear operator of the form
\[
\gen f(x, y) = \genX f(x,y) + \genY f(x, y), \quad f \in \domgen{\gen} \subset \Cinf{\stateXY},
\]
where $\genX$ acts on $x \mapsto f(x, y)$ and is the generator of $X$, and $\genY$ acts on $y \mapsto f(x, y)$ and for fixed $x$ is the generator of an affine process.
In the sequel, we outline the specifics of the operators $\genX$ and $\genY$ as well as of the operator domain $\domgen{\gen}$.

\paragraph*{The generator of $X$}
As suggested above, we characterize the process $X$ via the linear operator $\genX$ acting on functions in $\domgen\genX$ which is assumed to be a subset of $\Cinf{\stateX}$.
The restrictions we impose on the modulating process are fairly weak as we allow the generator of $X$ to be unbounded, extending the results of \cite{vanBeek2020RegimeFinance} to the case of diffusions and infinitely active jump processes.
Throughout we work with the following
\begin{assumption}\label{assumption:genX}
The operator $(\genX, \domgen\genX)$ with $\domgen{\genX} \subset \Cinf{\stateX}$ generates a Feller semigroup, denoted by $(\CsemiX_t)_{t \geq 0}$.
Moreover, $(\CsemiX_t)_{t \geq 0}$ is conservative and $\Ckc{2}{\stateX}$ is a core of $(\genX, \domgen\genX)$.
\end{assumption}

Assumption~\ref{assumption:genX}  is fulfilled by a large set of different Markov processes, for instance by conservative affine processes.
Note that the operator $\genX$ is not time-dependent, i.e.\ the Markov process $X$ is time-homogeneous.
Nevertheless, the extension to the inhomogeneous case is obvious by extending the process via $(t, X_t)_{t\geq 0}$.
The restriction that $(\CsemiX_t)_{t \geq 0}$ is conservative could be relaxed, albeit at the cost of increased technical complexity particularly in the proofs of Sections~\ref{sec:MP} and~\ref{sec:fourier}.
In the applications we have in mind there is anyhow no need for the modulating process to have finite life time.

\begin{remark}\label{remark:genX}
\cite{Schilling1998ConservativenessSemigroups} shows that under the above assumption, the semigroup $(\CsemiX_t)_{t \geq 0}$ has an extension to $\Cb{\stateX}$, which is $C_b$-Feller (cf.\ Definition~\ref{def:Cb_feller}).
We avoid confusions by denoting the corresponding semigroup on $\Cb{\stateX}$ with $(\CbsemiX_t)_{t \geq 0}$.
Moreover, we associate with the $C_b$-Feller semigroup a \emph{weak generator} $\weakgenX$ via
\begin{align}
    \label{eqn:weakgenXdom}
    \domgen\weakgenX &= \left\{f \in \Cb{\stateX} \mid \exists g  \in \Cb{\stateX} \text{ such that } \bplim{t \to 0+} \frac{\CbsemiX_t f - f}{t} = g\right\}, \\
     \label{eqn:weakgenX}
     \weakgenX f &= \bplim{t \to 0+} \frac{\CbsemiX_t f - f}{t}, \quad f \in \domgen\weakgenX.
\end{align}
Note that $\weakgenX|_{\domgen\genX} = \genX$, that $\domgen\weakgenX$ is dense with respect to the topology of weak convergence on the Banach space $(\Cb{\stateX}, \norm{\cdot}_\infty)$ and that $\weakgenX$ is closed with respect to that topology, cf.\ \cite[Chapter~I]{Dynkin1965MarkovProcesses}.
We frequently resort to the operator $(\weakgenX, \domgen\weakgenX)$, when  dealing with functions which are not in $\Cinf{\stateX}$.
\end{remark}

In~\cite{Duffie2003AffineFinance}, it is shown that the coefficients of the infinitesimal generator of a generic  affine process $ Y$ are  affine functions of $y$ that can be described  by a set of \emph{admissible} parameters.  For MMAPs  the constant part of these affine functions may depend on the state $x$ of the modulating process.
This  specification ensures the tractability of the characteristic function (see Section~\ref{sec:fourier} below).
In the following definition we introduce the Markov-modulated analogue of admissible parameters; their precise role will be clear in conjunction with Equation~\eqref{eqn:defGenY} below.
\begin{definition}\label{def:xadmissible}
We say that the parameters
\[
(a, \alpha, b, \beta, c, \gamma, m, \mu) = (a(x), \alpha, b(x), \beta, c(x), \gamma, m(x), \mu), \quad x \in \stateX,
\]
are \emph{$x$-admissible}, if for each fixed $x \in \stateX$
\begin{itemize}
    \item $a(x) \in \operatorname{Sem}^n$ with $a_{\indexI \indexI}(x) = 0$, and we have $a_{ij}(\cdot) \in \Bb{\stateX}$  for all $i,j \in \indexI \cup \indexJ$,
    \item $\alpha = (\alpha_1, \dots, \alpha_m)$ with each $\alpha_i \in \operatorname{Sem}^n$ and $\alpha_{i;\indexI(i) \indexI(i)} = 0$ for all $i \in \indexI$,
    \item $b(x) \in \stateY$ with $b_i(\cdot) \in \Bb{\stateX}$ for all $i \in \indexI \cup \indexJ$
    \item $\beta \in \R^{n \times n}$ such that $\beta_{\indexI \indexJ} = 0$ and $\beta_{i \indexI(i)} \in \R^{m-1}_+$ for all $i \in \indexI$,
    \item $c(x) \in \R_+$ with $c(\cdot) \in \Bb{\stateX}$,
    \item $\gamma \in \R^m_+$,
    \item $m(x, \cdot)$ is a Borel measure on $\stateY \backslash \{0\}$ with $\norm{M}_\infty := \sup_{x \in \stateX} M(x, \stateY \backslash \{0\}) < \infty$, where 
    \[
    M(x, d \xi) := (\innerp{\chi_\indexI(\xi)}{\bm{1} } + \norm{\chi_\indexJ(\xi)}^2) m(x, d\xi),
    \]
    and which satisfies for $f \in \Cb{\stateY}$ that
    \[
        \stateX\ni x\mapsto \int_{\stateY \backslash \{0\}} f(\xi) M(x, d \xi) \in \Bb{\stateX}, 
    \]
        \item $\mu = (\mu_1, \dots, \mu_m)$ where $\mu_i$ is a Borel measure on $\stateY\backslash \{0\}$ with $\mathcal{M}_i := \mathcal{M}_i(\stateY\backslash\{0\}) < \infty$ where
    \[
    \mathcal{M}_i(d\xi) := (\innerp{\chi_{\indexI(i)}(\xi)}{\bm{1}} + \norm{\chi_{\indexJ(i)}(\xi)}^2) \mu_i(d\xi).
    \]
\end{itemize}
The parameters are called \emph{strongly $x$-admissible}, if in addition $(a(x), b(x), c(x))$ are continuous in $x$, and if for $f \in \Cb{\stateY}$,
    \[
        \stateX\ni x\mapsto \int_{\stateY \backslash \{0\}} f(\xi) M(x, d \xi) \in \Cb{\stateX}.
    \]
\end{definition}
\cite{Filipovic2005Time-inhomogeneousProcesses} provides an accessible illustration of the above parameter conditions and their implications.

In this paper we define a MMAP $(X,Y)$ in terms of the generator of its associated semigroup. The following definition is needed to specify the domain of that operator.
For $f \in \Ck{2}{\stateXY}$ we set
\begin{align}
    \label{eqn:pregenXY1}
    [f]_1(x, y) &:= (1 + \norm{y^+})\left(|f(x, y)| + \norm{\nabla_y f(x, y)} + \sum_{k,l = 1}^n \left| \frac{\partial^2 f(x, y)}{\partial y_k \partial y_l}\right| \right),\\
    \label{eqn:pregenXY2}
    [f]_2(x, y) &:= |\innerp{y^*}{\beta^* \nabla_\indexJ f(x, y)}|,
\end{align}
where $\beta^* := (\beta^\transpose)_{\indexJ\indexJ} \in \R^{(n - m) \times (n - m)}$.
The advantage of the functions $[f]_1$ and $[f]_2$ in handling affine semigroups is best seen in Lemma~6.1 of \cite{Filipovic2005Time-inhomogeneousProcesses} and in Lemma~8.1 of \cite{Duffie2003AffineFinance}.

\paragraph*{The ``generator'' of $(X, Y)$} Consider for $x$-admissible parameters $(a, \alpha, b, \beta, c, \gamma, m, \mu)$ the following partial integro differential operator acting on $g \in \Ckb{2}{\stateY}$,
\begin{equation}
\begin{aligned}
        \genY g(y) &= \sum_{k,l=1}^n a_{kl}(x) \frac{\partial^2 g(y)}{\partial y_k \partial y_l} + \innerp{b(x)}{\nabla g(y)} - c(x) g(y) \\
    &\quad + \int_{\stateY\backslash \{0\}} \! \! (g(y + \zeta) - g(y) - \innerp{\nabla_\indexJ g(y)}{\chi_\indexJ(\zeta)})m(x, d \zeta)  \\
    &\quad + \sum_{k,l=1}^n \innerp{\alpha_{\indexI,kl}}{y^+}\frac{\partial^2 g(y)}{\partial y_k \partial y_l} +  \innerp{\beta y}{\nabla g(y)} - \innerp{\gamma}{y^+}g(y) \\
    &\quad + \sum_{i = 1}^m\int_{\stateY\backslash \{0\}} \!\! (g(y + \zeta) - g(y) -
     \innerp{\nabla_{\indexJ(i)} g(y)}{\chi_{\indexJ(i)}(\zeta)})y^+_i \mu_i(d \zeta).
\end{aligned}
\label{eqn:defGenY}
\end{equation}

Throughout we consider the linear operator $(\gen, \domgen{\gen})$ with
\begin{subequations}\label{eqn:defGenXY}
\begin{equation}
    \gen f(x, y) := \genX f(x, y) + \genY f(x, y)
\end{equation}
where $f$ is an element of
\begin{equation}
        \domgen{\gen} := \left\{f \in \Ck{2}{\stateXY} \cap \Cinf{\stateXY} \, \middle|
          \begin{aligned}
  &\, f(\cdot, y) \in \domgen{\genX} \text{ for all } y \in \stateY,\\
  &\, [f]_1, [f]_2 \in \Cinf{\stateXY}
  \end{aligned}
 \right\}
\end{equation}
\end{subequations}

\begin{definition}\label{def:mmaffine}
A Markov process $(Z, (\Pp_{z \in \stateXY}))$ with sub-Markov semigroup $(\semi_t)_{t \geq 0}$ on $\Bb{\stateXY}$ and  generator $(\genfull, \domgen\genfull)$ is called a  \emph{Markov-modulated affine process (MMAP)} if $\genfull|_{\domgen\gen} = \gen$. Moreover, we call $(\semi_t)_{t \geq 0}$ a \emph{Markov-modulated affine semigroup}.

If the parameters underlying $\gen$ are strongly $x$-admissible,  we refer to $(Z, \allowbreak (\Pp_z)_{z \in \stateXY})$ and $(\semi_t)_{t \geq 0}$ as a \emph{strongly regular Markov-modulated affine process} and as  a \emph{strongly regular Markov-modulated affine semigroup}, respectively.
\end{definition}

Definition~\ref{def:mmaffine} differs from the classical definition of affine processes in  \cite{Duffie2003AffineFinance}. 
There, a Markov process $Y$ is called affine if its characteristic function for fixed $t \geq 0$ is of the form $e^{\phi(t, u) + \innerp{\psi(t, u)}{y}}$,  where the functions  $\phi$ and $\psi$ are assumed to satisfy a crucial regularity assumption.
This technical condition is both necessary and sufficient for the existence of an infinitesimal generator of an affine semigroup (see e.g.\ Example~1.25.g and the surrounding discussion of \cite{bottcher2014levy} in the context of Feller processes).
\cite{Keller-Ressel2011AffineRegular} show that this regularity assumption is in fact superfluous, and further prove that stochastically continuous affine processes are Feller, cf.\ \cite[Theorem~3.5]{Keller-Ressel2011AffineRegular}.
Thus, the results of~\cite{Keller-Ressel2011AffineRegular} show that an affine process can equivalently be  characterized in terms of its generator, so that MMAPs as defined in this paper are a genuine extension of standard affine processes.

In general, $(\gen, \domgen{\gen})$ is \emph{not} the generator of a strongly continuous semigroup on $\Cinf{\stateXY}$.
On the one hand, the possible discontinuities in the $x$-admissible parameters prevent us from applying the theory of Hille-Yosida (cf.\ \cite[Section~1.2]{bib:ethier-kurtz-86}) to construct MMAPs via analytical semigroup arguments; on the other hand, perturbation arguments are not readily applicable since $\genX$ may be unbounded.

In fact, a priori it is not even clear if for a given generator $\genX$ satisfying Assumption~\ref{assumption:genX} and a given set of $x$-admissible parameters a MMAP in the sense of Definition~\ref{def:mmaffine} exists. 
We address this issue via the associated martingale problem.

\begin{definition}
Consider a linear operator $(\genericgen, \domgen{\genericgen})$ with $\domgen{\genericgen} \subset \Cinf{\stateXY}$.
We say that $(\Pp_z)_{z \in \stateXY} \subset \probset{\Omega}$ \emph{solves the martingale problem associated with $\genericgen$}, in short $\{\Pp_z\} \in \MP{\genericgen, \domgen{\genericgen}}$, if for $f \in \domgen{\genericgen}$ and for all $z \in \stateXY$ the process
\[
f(Z_t) - f(z) - \int_0^t \genericgen f(Z_s) ds, \quad t \geq 0
\]
is a $\{\calF_t\}$-martingale, where $\{\calF_t\}$ is the $\Pp_z$-completion of $\{\calF^0_t\}$.
Moreover, the martingale problem is said to be \emph{well-posed}, if for all $z$ any two solutions $\Pp_z^1, \Pp_z^2 \in \MP{\genericgen, \domgen{\genericgen}}$ have the same finite-dimensional distributions.
\end{definition}


%% file: 3.tex
\section{The associated Martingale Problem}\label{sec:MP}

In this section we discuss the existence of solutions to the martingale problem associated with the linear operator $\gen = \genX + \genY$ defined in \eqref{eqn:defGenXY}.
For fixed $x \in \stateX$, it follows from \cite{Duffie2003AffineFinance} that an appropriate extension of $\genY|_{\Ckc{\infty}{\stateY}}$  generates an affine semigroup.
Thus, we can regard $\genX$ as a perturbation of $\genY$.
Standard results on martingale problems related to perturbations typically fall into two categories (cf.\ \cite[Section~4.10.]{bib:ethier-kurtz-86}):
either the perturbation is bounded (corresponding to $\genX$ being a bounded operator), or the two  operators are independent, that is $\genY$ is independent of $x$.
Hence, the main difficulties in dealing with $\MP{\gen, \domgen{\gen}}$ stem from the fact that our setup falls in neither of the two categories.
Moreover, the popular approach of showing that $(\gen, \domgen{\gen})$ satisfies the \emph{positive maximum principle} (cf.\ \cite[Theorem~4.5.4]{bib:ethier-kurtz-86}) to solve $\MP{\gen, \domgen{\gen}}$ is also not applicable in our setup since it would require $\gen$ to map into $C(\stateXY)$.

 We tackle these issues via an approximation argument: we approximate $\genX$ by a sequence of bounded operators $\{\genX_k\}_{k \in \N}$ and we show that the sequence of solutions to the martingale problem for the operator $\gen_k + \genY$, $k \in \N$, is tight and that every limit point solves the martingale problem for $\genX + \genY$.

We begin with two auxiliary results.

\begin{lemma}\label{lemma:domdenseCinf}
$\domgen{\gen}$ is a dense subset of $\Cinf{\stateXY}$.
\end{lemma}
\begin{proof}
As in the proof of \cite[Proposition~6.3.]{Filipovic2005Time-inhomogeneousProcesses}, we introduce the following set
\[
\Theta_0 := \left\{ h \in \Ck{\infty}{\stateY} \middle|
\begin{aligned}
  &\, h(y^+, y^*) = \textstyle{\int_{\R^{n-m}} e^{\innerp{(v, iq)}{(y^+, y^*)}} g(q) dq}, \\
  &\, v \in \C^m_{--}, \, g \in \Ckc{\infty}{\R^{n-m}}
\end{aligned} \right\}.
\]
We further set $\overline{\Theta}_0 := \{ g h \mid g \in \Ckc{2}{\stateX}, \, h \in \Theta_0\}$.
Pick an arbitrary $f = gh \in \overline{\Theta}_0$.
As $\Ckc{2}{\stateX} \subset \domgen{\genX}$, we have $f(\cdot, y) \in \domgen{\genX}$ for all $y \in \stateY$.

Note that $\Theta_0 \subset \Schwartz{n}$.
Thus, we have (with slight abuse of notation) for an arbitrary multiindex $\alpha \in \N_0^n$ that $(1 + \norm{y^+})h \in \Schwartz{n}$ and $\nabla^\alpha h \in \Schwartz{n}$ (cf.\ \cite[Section~V.2]{Werner2018Funktionalanalysis}).
Since $g$ has compact support we infer that $[f]_1, [f]_2 \in \Cinf{\stateXY}$ and thus, $\overline{\Theta}_0 \subset \domgen{\gen}$.
Since the linear span of $\overline{\Theta}_0$ is dense in $\Cinf{\stateXY}$ (see the proof of \cite[Proposition~6.3.]{Filipovic2005Time-inhomogeneousProcesses}), we arrive at the result.
\end{proof}
\begin{lemma}\label{lemma:dom_closed_square}
$\domgen{\gen}$ is an algebra.
\end{lemma}
\begin{proof}
Let $f, g \in \domgen{\gen}$.
Obviously, $fg \in \Ck{2}{\stateXY} \cap \Cinf{\stateXY}$.
Moreover, using the product rule we get the estimates (see also \cite[Section~8]{Duffie2003AffineFinance})
\[
    \norm{[fg]_1}_\infty \leq K \norm{[f]_1}_\infty \norm{[g]_1}_\infty, 
\]
for some constant $K$.
Thus, we have $[fg]_1, [fg]_2 \in \Cinf{\stateXY}$.

It remains to consider $\genX (fg)(\cdot, y)$ for arbitrary $y \in \stateY$.
Since $\Ckc{2}{\stateXY}$ is a core of $\domgen{\genX}$, we can find sequences $(f_k), (g_k) \in \Ckc{2}{\stateX}$ converging to $f(\cdot, y)$ and $g(\cdot, y)$ w.r.t.\ $\norm{\cdot}_\genX$.
Since the $f_k$ and $g_k$ are bounded it holds that $\lim_{k \to \infty} \norm{f_k g_k - f(\cdot, y) g(\cdot, y)}_\infty = 0$ with $\{f_k g_k\} \subset \Ckc{2}{\stateX}$ as $\Ckc{2}{\stateX}$ is an algebra.
And since the $\norm{\cdot}_\genX$-closure of $\Ckc{2}{\stateX}$ is $\domgen{\genX}$, the result follows.
\end{proof}


Next, we introduce an approximating operator sequence $\{\genX_k\}_{k \in \N} $  and we study the martingale problem associated with the operator $\gen_k= \genX_k + \genY$.  
Since $(\genX, \domgen{\genX})$ generates a strongly continuous semigroup by Assumption~\ref{assumption:genX}, the Hille-Yosida theorem  implies first that $\domgen{\genX}$ is dense in $\Cinf{\stateX}$, second that $\operatorname{Range}(\alpha \operatorname{Id} - \genX) = \Cinf\stateX$ for any $\alpha > 0$ and third that $\genX$ is dissipative, i.e.\ $\norm{\alpha f - \genX f} \geq \alpha\norm{f}$ for any $\alpha > 0$ and $f \in \domgen{\genX}$. 
Moreover, we consider for each $k$ the following \emph{Yosida approximation}
\begin{equation}
    \label{eqn:yosida}
    \genX_k := k \genX G_k,
\end{equation}
where $G_k = (k \Id - \genX)^{-1}$.
Note that the operator $G_k$, the resolvent of the semigroup $(\CsemiX_t)_{t \geq 0}$, exists by standard results from semigroup theory, see e.g.\ Section~VII.4 of \cite{Werner2018Funktionalanalysis}. 
Then,  $\{\genX_k\}_{k \in \N}$ is a sequence of \emph{bounded} operators, and it holds that $\lim_{k \to \infty} \genX_k f = \genX f$ (where the limit is pointwise for $f \in \domgen{\genX}$), since $(\genX, \domgen{\genX})$ generates a strongly continuous semigroup, see for instance \cite[Chapter~1]{bib:ethier-kurtz-86} for details.

\begin{lemma}\label{lemma:boundedPerturbation}
Consider the operator $(\gen_k, \domgen{\gen})$, where
\[
\gen_k f(x, y) := \genX_k f(x, y) + \genY f(x, y), \quad f \in \domgen{\gen}.
\]
Then, there exists a solution to the martingale problem associated with $(\gen_k, \domgen{\gen})$.
\end{lemma}
\begin{proof}
To prove the lemma we rely on \cite[Proposition 4.10.2]{bib:ethier-kurtz-86}.
Accordingly, we first show that for each $k$ there exists a kernel $\pi^X_k$ such that we obtain the following representation for any $g \in \domgen{\genX}$,
\begin{equation}
\label{eqn:repres_Riesz_proof_bounded}
    \genX_k g(x) = {k} \int_{\stateX} (g(\zeta) - g(x)) \pi^X_k(x, d \zeta), \quad \forall x\in {\stateX}.
\end{equation}
To establish \eqref{eqn:repres_Riesz_proof_bounded} note first  that the operator $G_k$ from~\eqref{eqn:yosida} is a bijection from $\Cinf{\stateX}$ onto $\domgen{\genX}$.
Additionally, since $G_k$ is a positive linear operator, we know by the Riesz representation theorem that there exists a kernel $\pi^X_k$ such that for any $f \in \Cinf{\stateX}$
\[
k G_k f(x) = \int_{\stateX} f(\zeta) \pi^X_k(x, d \zeta), \quad \forall x\in {\stateX}.
\]
Pick an arbitrary $g \in \domgen{\genX}$.
Consequently, since $(k\Id - \genX)g \in \Cinf{\stateX}$ by the Hille-Yosida theorem, it holds that
\[
k g(x) =(k \Id -\genX)k G_k  g (x) = \int_{\stateX} (k \Id - \genX) g(\zeta) \pi^X_k(x, d \zeta), \quad  x \in \stateX.
\]
We further have that
$
\int_{\stateX} \genX g(\zeta) \pi^X_k(\cdot, d \zeta) = \genX_k g,
$
and so overall we obtain the representation~\eqref{eqn:repres_Riesz_proof_bounded} for any $g \in \domgen{\genX}$.

Next we explain how to apply \cite[Proposition 4.10.2]{bib:ethier-kurtz-86}. First note that \eqref{eqn:repres_Riesz_proof_bounded} gives for $f \in B_b(\stateXY)$ that
$$\genX_k f(x, y) = \mathfrak{B}f(x,y):= \int_{\stateX} (f(\zeta,y) - f(x,y)) \pi^X_k(x, d \zeta)\,.$$
We define the operator $\mathfrak{A} \colon\domgen{\gen} \to B_b(\stateXY)$ by $\mathfrak{A}f(x,y) =  \genY f(x, y)$ (the fact that  $\genY f(x, y)$ is bounded for $f \in \domgen{\gen}$ is shown in the proof of Theorem~\ref{thm:MP} below). 
Clearly, for  $f \in \domgen{\gen}$ it holds that $\gen_k f = \mathfrak{B} f + \mathfrak{A} f$.

Fix $z=(x,y) \in \stateXY.$ A solution $\Pp_z \in \probset{\Omega}$ of the martingale problem for $\mathfrak{A}$ is a measure on $D_{\stateXY^\Delta}$ such that $\Pp_z (X_t =x \text{ for all } t \ge 0) =1$.
Moreover, under $\Pp_z$,  $Y$ is an affine process with $Y_0=y$ a.s. and generator $\genY $; such a measure exists by standard results on affine processes.  
\cite[Proposition 4.10.2]{bib:ethier-kurtz-86} now gives the existence of a solution of the martingale problem for $\mathfrak{B} + \mathfrak{A}$ and hence the claim.
\end{proof}

Note that in the proof of the above lemma, \eqref{eqn:repres_Riesz_proof_bounded} shows that in probabilistic terms, the approximation of $\genX$ by $\{\genX_k\}_{k \in \N}$ corresponds to the approximation of the Feller process $X$ by a sequence of pure jump processes.

\begin{theorem}\label{thm:MP}
Consider the linear operator $(\gen, \domgen{\gen})$ given by~\eqref{eqn:defGenXY}, where the underlying parameters $(a, \alpha, b, \beta, c, \gamma, m, \mu)$ are $x$-admissible and where $\genX$ satisfies Assumption~\ref{assumption:genX}.
Then there exists a solution to the martingale problem associated with $(\gen, \domgen{\gen})$.
\end{theorem}

\begin{proof}
For a closed subset $U \times V \subset \stateX \times \stateY$ and a function $f \in \domgen{\gen}$, we introduce
\[
\norm{f}_{\sharp; U \times V} := \sup_{(x, y) \in U \times V} \{[f]_1(x, y) + [f]_2(x, y) \},
\]
where $[f]_1$ and $[f]_2$ were introduced in \eqref{eqn:pregenXY1} and \eqref{eqn:pregenXY2}.
The proof is divided into two steps.
First, we show that the measures $\Pp^k_z \in \MP{\gen_k, \domgen{\gen}}, \, k \in \N$ form a tight collection.
In the second step, we collect several classical results to conclude that there exist limit points in $(\Pp^k_z)_{k \in \N}$ that solve $\MP{\gen, \domgen{\gen}}$.

\paragraph*{Step 1:\ Tightness}  Fix an arbitrary $z \in \stateXY$.
By Lemma~\ref{lemma:boundedPerturbation} there is for each $k$ a measure $\Pp_z^k \in \MP{\gen_k, \domgen{\gen}}$ with the operator $\gen_k$ acting on $f \in \domgen{\gen}$ via
\[
\gen_k f(x, y) = \genX_k f(x, y) + \genY f(x, y).
\]
To keep the notation lean, we omit the subscript $z$ in the sequel, i.e.\ we write $\Pp^k$ and $\E^k$ instead of $\Pp^k_z$ and $\E^k_z$, respectively.
We want  to show that $\{\Pp^k\}_{k \in \N}$ is tight.

Fix an arbitrary $f \in \domgen{\gen}$.
We use $\Pp^k_f$ to denote the distribution of the paths of $f(X,Y)$ (with corresponding expectation $\E_{f}^{k}$) and let $(\xi_t)_{t \geq 0}$ be the coordinate process of real-valued càdlàg functions, where the $\sigma$-algebra $\calF^{k,\xi}$ is the $\Pp^k_f$-completion of $\smash{\sigma(\xi_t  :  t \geq 0)}$. We denote the corresponding filtration with $\smash{\{\calF^{k,\xi}\}_t}$.
Since $\domgen{\gen}$ is dense in $\Cinf{\stateXY}$, we know from \cite[Theorem~3.9.1]{bib:ethier-kurtz-86} that tightness of $(\Pp^k)_{k \in \N}$ is equivalent to tightness of $(\Pp_f^k)_{k \in \N}$.
We introduce the following function,
\[
g\,:\,\Omega \to D_\R \, ; \, \omega \mapsto f \circ \omega.
\]
Let $\epsilon \in (0, 1]$.
Set $q(e, d) := |e-d| \wedge 1$, for any $e, d \in \R$.
Then, for every $\smash{A \in \calF^{k,\xi}_t, 0 \leq t < \infty}$ and $0 \leq h \leq \epsilon$ we have
\begin{align}
    \EVM{f}{k}{\indA{A} q(\xi_{t+h}, \xi_t)^2} &\leq \int_A (\xi_{t+h} - \xi_t)^2 d\Pp^k_f \notag \\
    &= \int_{g^{-1}(A)}(f(X_{t+h}, Y_{t+h}) - f(X_t, Y_t))^2 d \Pp^k \notag \\
    \label{eqn:MPexist1}
    &= \int_{g^{-1}(A)} \EVM{}{k}{(f(X_{t+h}, Y_{t+h}) - f(X_t, Y_t))^2 \mid \calF_t} d \Pp^k.
\end{align}
Since $\Pp^k \in \MP{\gen_k, \domgen{\gen_k}}$ and $f, f^2 \in \domgen{\gen_k}$ by Lemma~\ref{lemma:dom_closed_square}, we get
\begin{align*}
    &\EVM{}{k}{(f(X_{t+h}, Y_{t+h}) - f(X_t, Y_t))^2 \mid \calF_t} \\
    & \quad =  \EVM{}{k}{f^2(X_{t+h}, Y_{t+h}) - f^2(X_t, Y_t) \mid \calF_t}  \\
    &\qquad \qquad -2 f(X_t, Y_t)\EVM{}{k}{f(X_{t+h}, Y_{t+h}) - f(X_t, Y_t) \mid \calF_t} \\
    &\quad =\EVM{}{k}{\int_t^{t+h}\gen_k f^2(X_{s}, Y_{s}) ds \mid \calF_t}  -2 f(X_t, Y_t)\EVM{}{k}{\int_t^{t+h} \gen_k f(X_{s}, Y_{s})ds \mid \calF_t}
\end{align*}
Next, we handle the integrands.
For that recall that the specific form of $\gen_k f(x, y)$ equals
\begin{align}
    &\genX_k f(x, y) + \sum_{k,l=1}^n a_{kl}(x) \frac{\partial^2f(x,y)}{\partial y_k \partial y_l} + \innerp{b(x)}{\nabla_y f(x,y)} - c(x) f(x, y) \notag \\
    &\quad + \int_{\stateY\setminus \{0\}} (f(x, y + \zeta) - f(x, y) - \innerp{\nabla_\indexJ f(x, y)}{\chi_\indexJ(\zeta)})m(x, d \zeta) \notag \\
        \label{eqn:existMPaffine1}
    &\quad + \sum_{k,l=1}^n \innerp{\alpha_{\indexI,kl}}{y^+}\frac{\partial^2f(x,y)}{\partial y_k \partial y_l} +  \innerp{\beta y}{\nabla_y f(x,y)} - \innerp{\gamma}{y^+}f(x, y) \\
                \label{eqn:existMPaffine2}
    &\quad + \sum_{i = 1}^m\int_{\stateY\setminus \{0\}} (f(x, y + \zeta) - f(x, y) -
   \innerp{\nabla_{\indexJ(i)} f(x, y)}{\chi_{\indexJ(i)}(\zeta)})y_i^+ \mu_i(d \zeta). 
\end{align}
We compactly write the expressions~\eqref{eqn:existMPaffine1}~to~\eqref{eqn:existMPaffine2} as $\Tilde{\gen}^Y f(x, y)$ and note that $\Tilde{\gen}^Y$ is the generator of an affine semigroup with admissible parameters $\smash{(0, \alpha, 0, \beta, 0, \gamma, 0, \mu)}$.
Thus, by Lemma~8.1 of \cite{Duffie2003AffineFinance} we know that for $(x, y) \in \stateXY$ there exists a constant $C$ such that
\begin{equation*}
    \Tilde{\gen}^Y f(x, y) \leq C \left(\norm{\alpha} + \norm{\beta} + \norm{\gamma} + \sum_{i \in \indexI} \mathcal{M}_i \right) \norm{f(x, y + \cdot)}_{\sharp; \stateY},
\end{equation*}
and since $f \in \Cinf{\stateXY}$ this in particular implies that there is a constant $C^Y_f$ such that
\begin{equation}
    \label{eqn:existMPbound}
    \Tilde{\gen}^Y f(x, y) \leq C_f^Y,
\end{equation}
with the constant $C^Y_f$ depending on $f$, but not on $k$.
Furthermore, with $\genX_k$ being a bounded operator, there exists a constant $C_{k,f}^X$ (depending both on $f$ and $k$) such that
\begin{equation}
    \label{eqn:existMPbound2}
     \genX_k f(x, y) \leq C_{k,f}^X.
\end{equation}
For the remaining terms we yet again apply Lemma~8.1 of \cite{Duffie2003AffineFinance} and exploit the structure of $\domgen{\gen_k}$ to arrive at an inequality of the form
\begin{align}
    &\gen_kf(x, y) - (\genX_k f(x, y) + \Tilde{\gen}^Y f(x, y)) \\
    &\qquad \leq C \left(\norm{(\norm{a_{kl}}_\infty)_{k,l = 1}^n} + \norm{(\norm{b_{k}}_\infty)_{k = 1}^n} + \norm{c}_\infty + \norm{M}_\infty \right) \norm{f(x, y + \cdot)}_{\sharp; \stateY} \notag\\
    \label{eqn:existMPbound3}
    &\qquad \leq C_f^{Y|X},
\end{align}
for a constant $C_f^{Y|X}$ independent of $k$.
Similar considerations also apply for $\gen_k f^2$ and so the inequalities~\eqref{eqn:existMPbound} to~\eqref{eqn:existMPbound3} imply
\begin{align}
    &\EVM{}{k}{(f(X_{t+h}, Y_{t+h}) - f(X_t, Y_t))^2 \mid \calF_t} \notag\\
    &\quad \leq \epsilon (C_{k,f^2}^X + C^{Y|X}_{f^2} + C^Y_{f^2}) + 2 \epsilon
        \label{eqn:existMPbound_final}\norm{f}_\infty ( C_{k,f}^X + C^{Y|X}_{f} + C^Y_{f}) \quad \Pp^k\text{-a.s.}
\end{align}
It remains to consider the behavior of the constant $C_{k,f}^X$ when taking the supremum over $k$.
However, for a Yosida approximation it holds that $\genX_k f$ converges uniformly to $\genX f$, which in turn implies the finiteness of $\sup_k \norm{\genX_k f}_\infty$ (uniformly convergent bounded functions are uniformly bounded).
Consequently, we can find a constant independent of $k$ and we have thus shown that $\lim_{t \to 0}\sup_{k \in \N}\EVM{f}{k}{q(\xi_0, \xi_t)^2} = 0$.
By defining $\gamma_\epsilon$ as the right-hand side of~\eqref{eqn:existMPbound_final}, we can regard $(\gamma_\epsilon)_{\epsilon \in (0, 1]}$ as a family of functions $\gamma_\epsilon \, : \, D_\R \to [0, \infty)$ with the following properties
\begin{align*}
    &\sup_{k \in \N} \EVM{f}{k}{\gamma_\epsilon} \, \to \, 0 \quad \text{as } \epsilon \, \to \, 0, &\\
    &\EVM{f}{k}{q(\xi_{t+ h}, \xi_t)^2 \, q(\xi_t, \xi_{t-v})^2 \mid \calF^{k,\xi}_t} \leq \EVM{f}{k}{\gamma_\epsilon \mid \calF^{k,\xi}_t} \quad \Pp_f^k\text{-a.s.}&
\end{align*}
for all $0 \leq t \leq T$, $0 \leq h \leq \epsilon \leq 1$ and $0 \leq v \leq \epsilon \wedge t$.
We work on $D_{\stateXY^\Delta}$, so that the compact containment condition is trivially fulfilled.
Hence, by Theorem~3.8.6 of \cite{bib:ethier-kurtz-86} we conclude that $(\Pp_f^k)_{k \in \N}$ is tight for all $f \in \domgen{\gen}$.
By Theorem~3.9.1 of \cite{bib:ethier-kurtz-86} together with Lemma~\ref{lemma:domdenseCinf} this implies that $(\Pp^k)_{k \in \N}$ is tight.

\paragraph*{Step 2:\ Limit} Since $(\Omega, \topology)$ is a Polish space, by Prokhorov's Theorem the set $(\Pp^k)_{k \in \N} \subset \probset{\Omega}$ is relatively compact.
So, there exists an accumulation point of $\{\Pp^k\}$, which we denote by $\Pp$ (with corresponding expectation $\E$).
In the sequel we show that $\Pp \in \MP{\gen, \domgen{\gen}}$, which is equivalent to showing that for all $f \in \domgen{\gen}$ it holds that
\begin{equation}
    \label{eqn:MPequivalence}
    \EV{\left(f(X_{t_2}, Y_{t_2}) - f(X_{t_1}, Y_{t_1}) - \int_{t_1}^{t_2} \gen f(X_s, Y_s) ds  \right) \, \prod_{l = 1}^m h_l(X_{s_l}, Y_{s_l}) } = 0,
\end{equation}
for all $0 \leq s_1 < \dots  s_m \leq t_1 \leq t_2$ and all $h_1, \dots, h_m \in \Bb{\stateXY}$.
By Lemma~3.7.7 of \cite{bib:ethier-kurtz-86} the following dense subset of $[0, \infty)$,
$$T_\Pp := \{t \geq 0 \mid \Pp((X_{t-}, Y_{t-}) = (X_t, Y_t)) = 1\},$$
has an at most countable complement in $[0, \infty)$.
We pick a subsequence $\{\Pp^{k(j)}\}$ for which \smash{$\lim_{j \to \infty} \Pp^{k(j)} = \Pp$}.
Then, by \cite[Theorem~3.7.8]{bib:ethier-kurtz-86} the finite dimensional distributions of $((X_{t_1}, Y_{t_1}),\allowbreak \dots, (X_{t_m}, Y_{t_m}))$ under $\Pp^{k(j)}$ converge weakly to the corresponding distributions under $\Pp$ for $t_1,  \dots, t_m \in T_\Pp$.
Thus, it is enough to show~\eqref{eqn:MPequivalence} for $s_1, \dots, s_m, t_1, t_2 \in T_\Pp$.
Since $\Pp^{k(j)} \in \MP{\gen_{k(j)}, \domgen{\gen}}$, it is left to prove that
\begin{align}
\label{eqn:MP_convergence}
    &\lim_{j \to \infty} \EVM{}{k(j)}{\int^{t_2}_{t_1} \gen_{k(j)} f(X_s, Y_s) ds \, \prod_{l=1}^m h_l(X_{s_l}, Y_{s_l}) } \\
    &\qquad \qquad \qquad \qquad = \EV{\int^{t_2}_{t_1} \gen f(X_s, Y_s) ds \, \prod_{l=1}^m h_l(X_{s_l}, Y_{s_l})} \notag
\end{align}
Given the uniform convergence of $\{\gen_{k(j)} f\}$ we have
\begin{equation}
\label{eqn:MP_convergence1}
    \lim_{j \to \infty} \EVM{}{k(j)}{(\gen_{k(j)} f(X_s, Y_s) - \gen f(X_s, Y_s) ) \, \prod_{l = 1}^m h_l(X_{s_m}, Y_{s_m}) } = 0,
\end{equation}
for all $s \in T_\Pp$.
The weak convergence of $\{\Pp^{k(j)}\}$ further implies that
\begin{equation}
\label{eqn:MP_convergence2}
    \lim_{j \to \infty} \EVM{}{k(j)}{\gen f(X_s, Y_s) \, \prod_{l = 1}^m h_l(X_{s_l}, Y_{s_l}) } = \EV{\gen f(X_s, Y_s) \, \prod_{l = 1}^m h_l(X_{s_l}, Y_{s_l}) },
\end{equation}
and so Equations~\eqref{eqn:MP_convergence1} and~\eqref{eqn:MP_convergence2} prompt
\[
     \lim_{j \to \infty} \EVM{}{k(j)}{\gen_{k(j)} f(X_s, Y_s) \, \prod_{l = 1}^m h_l(X_{s_l}, Y_{s_l}) } = \EV{\gen f(X_s, Y_s) \, \prod_{l = 1}^m h_l(X_{s_l}, Y_{s_l}) },
\]
Finally, we interchange order of integration in~\eqref{eqn:MP_convergence} to arrive at the desired result.
\end{proof}

In Sections~\ref{sec:fourier} and~\ref{sec:further-properties}, we frequently resort to the following useful extension of Theorem~\ref{thm:existence}.
Recall that we introduced $\genY$ as an operator on $\Ckb{2}{\stateY}$.
The proof is postponed to Appendix~\ref{appendix:proof}.
\begin{corollary}\label{cor:afterMPresult}
For an arbitrary $z \in \stateXY$, choose a $\Pp_z \in \MP{\gen, \domgen{\gen}}$.
Let $f \in C([0, \infty) \times \stateX)$, which is $C^1$ in its first argument and $f(t, \cdot) \in \domgen{\weakgenX}$ for all $t \geq 0$.
Moreover, consider $g \in \Ck{1,2}{[0, \infty) \times \stateY}$, which is bounded in its second argument.
Then, the process
\[
f(t, X_t)g(t, Y_t) - f(0, x)g(0,y) - \int_0^t (\partial_t + \weakgenX + \genY)f(s, X_s)g(s, Y_s)ds, \quad t \geq 0,
\]
is a martingale under $\Pp_z$.
\end{corollary}

%% file: 4.tex
\section{Transform Formula and Existence}\label{sec:fourier}

The popularity of affine processes largely stems from  the fact that the Fourier transform of their marginal distributions is available in semi-explicit form (up to the solution of a system of ODEs). 
In particular, in the context of affine models many pricing problems in mathematical finance can be solved efficiently by Fourier methods. 
In this section we show that the characteristic function of MMAPs has  a fairly simple form as well, which makes these processes an appealing tool for many modelling tasks. 
On the theoretical side,  our results on the characteristic function of MMAPs permit us to establish the \emph{uniqueness} of the martingale problem associated with the generator $\gen$ from~\eqref{eqn:defGenXY} and hence the \emph{existence} of an MMAP.

Consider some $u \in \C^n$ and the function
\begin{equation}
    \label{eqn:exp_function}
    \stateY \to \C \, : \, y \mapsto e^{\innerp{u}{y}}.
\end{equation}
The function~\eqref{eqn:exp_function} is an element of $\Cb{\stateY}$ if and only if $u$  belongs to the set
\begin{equation*}
    \U := \C^m_{-} \times i\R^{m - n}.
\end{equation*}
Below we discuss expectations of the form $\EVM{z}{}{e^{\innerp{u}{Y_t}}}$ for $u \in \U$ and any $z \in \stateXY, t \geq 0$; this includes the characteristic function of $Y_t$ as a special case.

For an  affine semigroup $(P^{\text{aff}}_t)$, it holds that $P^{\text{aff}}_t e^{\innerp{u}{y}} = \exp(\phi(t, u) + \innerp{\psi(t, u)}{y})$ for $u \in \U$ and deterministic functions $\phi$ and $\psi$. 
The derivatives of $\phi(\cdot, u)$ and $\psi(\cdot, u)$ at $t=0$ are typically denoted by $F(u)$ and $R(u)$, respectively (their existence is shown in~\cite{Keller-Ressel2011AffineRegular}). 
In what follows we introduce  the Markov-modulated analogue of these derivatives.
Consider functions $F : \stateX \times \U \to \C$ and $R : \U \to \C^n$ given by
\begin{align}
        \label{eqn:Riccati_F}
        \begin{split}
                F(x, u) &= \innerp{b(x)}{u} + \innerp{a(x) u}{u} - c(x) \\
    &\qquad + \int_{\stateY \backslash \{0\}} (e^{\innerp{u}{\zeta}} - 1 - \innerp{u_\indexJ}{\chi_\indexJ(\zeta)})m(x, d\zeta),
        \end{split}  \\
        \label{eqn:Riccati_Ri}
    R_{i} (u) &= \innerp{\alpha_i u}{u} + \innerp{\beta^+_i}{u} - \gamma_i + \int_{\stateY \backslash \{0\}} (e^{\innerp{u}{\zeta}} - 1 - \innerp{u_{\indexJ(i)}}{\chi_{\indexJ(i)}(\zeta)}) \mu_i(d \zeta), \\
        \label{eqn:Riccati_RJ}
    R_\indexJ (u)&= \beta^* u^*.
\end{align}
for $i \in \indexI$ and where
\begin{align*}
    \beta^+_i &:= (\beta^\transpose)_{i \{1, \dots, n\}} \in \R^n, \quad i \in \indexI, \\
    \beta^* &:= (\beta^\transpose)_{\indexJ\indexJ} \in \R^{(n - m) \times (n - m)}.
\end{align*}
Note  that $R$ is the same function as in the standard affine case, whereas $F$ is modified (it depends  on $x$). Moreover, $F$ and $R$ clearly separate the $x$-admissible parameters according to whether they depend on $x$ or not.
In the standard affine case, $\psi$ is characterized by a system of ODEs, called generalized Riccati equations, and then $\Phi(t,u) = \exp(\phi(t,u))$ is simply given via the linear ODE $\partial_t \Phi(t, u) = \Phi(t, u) F(\psi(t, u))$.  
It will turn out that for an MMAP  the transform $\EVM{z}{}{e^{\innerp{u}{Y_t}}}$ has a similar structure, but the function $\Phi(t,u)$ is replaced by a function ${\varphi}(t, x; u)$ that is closely connected to a Cauchy problem involving the function $F$ and the generator of $X$.

Consider a solution $\psi$ of the generalized Riccati equations mentioned before, that is a function $\psi : [0, \infty) \times \U \to \C^n$ solving the ODE system
\begin{align}
    \label{eqn:riccati}
    \partial_t \psi (t, u) &= R (\psi (t, u)) , \quad \psi (0, u) = u.
\end{align}
The system~\eqref{eqn:riccati} is a well-studied mathematical object and \cite[Section~6]{Duffie2003AffineFinance} prove the existence of a unique solution.
Moreover, Equation~\eqref{eqn:riccati} with \eqref{eqn:Riccati_RJ} obviously implies that $\psi_\indexJ(t, u) = e^{\beta^* t} u^*$.
Within our framework, the counterpart of the linear ODE for $\Phi$ is the following  Cauchy problem for a generic operator $\gengenX$, which  will be either equal to $\genX$ or to $\weakgenX$,
\begin{equation}
    \label{eqn:cauchy_prime}
    \begin{cases}
    &\partial_t {\varphi^\star}(t, x; u) = \gengenX{\varphi^\star}(t, x; u) + {\varphi^\star}(t,x;u) F(x, \psi(t, u)), \,\, t \in [0, T], \, x \in \stateX, \\ 
    &\varphi^\star(0, x; u) = f(x), \quad f \in \domgen{\gengenX}, 
    \end{cases} \tag{$\text{CP}_{T, u}$}
\end{equation}
where $T > 0$ and $u \in \U$ are to be understood as exogenous parameters. 
We recall the notion of an important solution concept in the following
\begin{definition}\label{def:solutionsCauchy}
A \emph{classical solution} of \eqref{eqn:cauchy_prime} is a $C^1$-function $[0,T] \to \domgen{\gengenX} : t \mapsto {\varphi^\star}(t, \cdot; u)$  satisfying \eqref{eqn:cauchy_prime}.
\end{definition}

To motivate  our results on the form of the function $ \EVM{z}{}{e^{\innerp{u}{Y_{t}}}}$ for MMAPs we begin with the case where there exists a solution $\varphi^\star$ to~\eqref{eqn:cauchy_prime} with $\gengenX = \weakgenX$ and initial condition $f\equiv 1$.
Fix arbitrary $u \in \U,\, y \in \stateY, \, t^\star \in [0, T]$, and define the function $G_u \colon [0,T] \times \stateXY \to \C $ by $G_u(t,x,y) = \varphi^\star(t, x; u) e^{\innerp{\psi(t, u)}{y}}$.  
We want to use Corollary~\ref{cor:afterMPresult} to show that
\begin{equation}
\label{eqn:char_proof1a}
M_t = G_u(t^\star -t, X_t, Y_t), \quad t \in [0, t^\star],
\end{equation}
is a complex-valued martingale. 
Note first that
\begin{equation}\label{eqn:char_proof1c}
    \genY G_u(t, y) =   G_u(t, x, y) \left( F(x, \psi(t, u)) + \innerp{R(\psi(t, u))}{y} \right).
\end{equation}
Using that  $\varphi^\star$ solves Cauchy problem \eqref{eqn:cauchy_prime} with $\gengenX = \weakgenX$ and the generalized Riccati equation~\eqref{eqn:riccati} bring us to
\begin{equation}\label{eqn:PDEcombGu}
    \begin{aligned}
        (\weakgenX &+ \genY- \partial_t) G_u(t, x, y) \\&= e^{\innerp{\psi(t, u)}{y}}(\weakgenX - \partial_t) \varphi^\star(t, x; u)  - \innerp{\partial_t \psi(t, u)}{y} G_u(t, x, y) + \genY G_u(t, x, y)  \\
    &= -G_u(t, x, y)(F(x, \psi(t, u)) - \innerp{R(\psi(t, u))}{y})  + \genY G_u(t, x, y).
    \end{aligned}
\end{equation}
Together with Equation~\eqref{eqn:char_proof1c}, we consequently get
\[
(\partial_t + \weakgenX + \genY) G_u(t^\star - \cdot, \cdot) \equiv 0.
\]
Thus, $M$  is a martingale by Corollary~\ref{cor:afterMPresult}. 
Using that $G_u(0,x,y) = e^{\innerp{u}{y}} $ we get that
\begin{equation}\label{eq:char-func-regular}
\EVM{z} {}{e^{\innerp{u}{Y_{t^\star}}}} =  \EVM{z} {} {M_{t^\star}} = M_0 = \varphi^\star(t^\star, x; u) e^{\innerp{\psi(t^\star, u)}{y}},
\end{equation}
that is we have identified the characteristic function of $Y_{t^\star}$.
It is well known that regularity assumptions on the function $F(\cdot,u)$ (in particular continuity) are needed to ensure that the Cauchy problem \eqref{eqn:cauchy_prime} has a classical solution; see for instance~\cite{Pazy1983SemigroupsEquations} or~\cite{Glau2016ARates}. 
By the Feynman-Kac Theorem (see e.g.\ \cite[Theorem~17.4.10]{cohen2015stochastic}) we have for $u \in \U$ the representation
\begin{equation}
\varphi^\star(t, x; u) =  \varphi(t,x;u) :=  \EVM{(x,y)}{}{\exp\left(\int_0^t F(X_s, \psi(t - s, u)) ds\right)}.
\end{equation}
Note that the function $\varphi(\cdot;u)$ is well-defined, even if~\eqref{eqn:cauchy_prime} does not admit a classical solution.
In Theorem~\ref{prop:fourier} below we use  approximation arguments to show that for general $x$-admissible parameters the characteristic function of $Y_{t^\star}$ is equal to
\[
 \varphi(t^\star, x; u) e^{\innerp{\psi(t^\star, u)}{y}}.
\]
We begin with an important boundedness result.
Its proof is postponed to the appendix.

\begin{lemma}\label{lemma:boundednessF}
Consider the function $F$ defined in \eqref{eqn:Riccati_F}.
Then for any compact subset $\mathcal{V} \subset \U$ it holds that
\[
\sup_{(x, u) \in \stateX \times \mathcal{V}} |F(x,u)| < \infty.
\]
\end{lemma}

Consider a measure $\Pp_z \in \MP{\gen, \domgen{\gen}}$. 
In the next result we derive the characteristic function of the one-dimensional distributions of the process $Y$ under $\Pp_z$.

\begin{theorem}\label{prop:fourier}
Let $z \in \stateXY$, $t^\star \geq 0$ and $u \in \U$ be arbitrary and consider $\Pp_z \in \MP{\gen, \allowbreak \domgen{\gen}}$.
Then it holds that
\begin{equation}
    \label{eqn:fourier}
    \EVM{z}{}{e^{\innerp{u}{Y_{t^\star}}}} = \varphi(t^\star, x; u) e^{\innerp{\psi(t^\star, u)}{y}}, 
\end{equation}
where $\psi(\cdot, u) : [0, \infty) \to \C^n$ solves the system of generalized Riccati equations~\eqref{eqn:riccati} and where $\varphi(t, x; u) = \EVM{z}{}{\exp(\int_0^t F(X_s, \psi(t - s, u)) ds)}, \, t \in [0, t^\star]$ (the value of the expectation is invariant with respect to $y$).
\end{theorem}

\paragraph*{Comments} 
The product structure in \eqref{eqn:fourier} is very convenient for computing Fourier-Laplace transforms. 
In a first step one computes the function $\psi(t,u)$ by solving the generalized Riccati equations~\eqref{eqn:riccati}, analogously as in the standard affine case. 
In a second step the function $\varphi$ is either computed using Monte Carlo techniques or, under the added regularity of Proposition~\ref{prop:fourier2}, as a solution of the Cauchy problem~\eqref{eqn:cauchy_prime}. 
The complexity of this step depends on the form of the generator of $X$: if $X$ is a finite state Markov chain, then \eqref{eqn:cauchy_prime} reduces to a set of ODEs; if $X$ is a diffusion process, it is a linear PDE of parabolic type, etc.

\begin{proof}[Proof of Theorem~\ref{prop:fourier}.]  The proof is based on an approximation argument and proceeds in two steps.
\paragraph*{Step 1} Fix an arbitrary $u\in \U, \bar x \in \stateX$. 
In this step we show that there is a sequence $\{F^k_u\}, \; k \in \N$  of functions in  $\Ckb{1}{\stateX \times[0, t^\star]}$ with $\sup_{k \in \N} \norm{F_u^k}_\infty < \infty$ and such that
\begin{equation}\label{eqn:F-approximation}
\lim_{k \to \infty}  {\int_0^{t^\star} \big | F_u^k(X_s, {t^\star}-s) - F(X_s, \psi({t^\star}-s,u)) \big | ds } =0 \quad \Pp_{\bar x} \text{-a.s.}
\end{equation}
Define a measure $\mu_{\bar x}$ on the Borel sets of $\stateX \times [0,{t^\star}]$ by
$
\mu_{\bar x} (A) = \EVM{\bar x}{} {\int_0^{{t^\star}} \indA{A} (X_s,s) ds }\,.
$
We use Lemma~\ref{lemma:boundednessF} to deduce that for  $u \in \U$ the function $F(\cdot,u) \colon (x,t)\mapsto F(x,\psi(t,u))$ is bounded.
Moreover, it follows from \cite[Theorem~1]{Wisniewski1994TheSpaces} that $F(\cdot,u)$ is $\mu_{\bar x}$-a.e.~the pointwise limit of uniformly bounded and {continuous} functions $\widetilde{F}^k_u $ on $(\stateX \times [0, t^\star])$.    
Using standard approximation arguments for continuous functions we may approximate  the functions $\widetilde{F}^k_u $ locally uniformly by  $C^1$ functions, so that there is a sequence  $F^k_u \in  C^1_b (\stateX \times[0, t^\star])$ of uniformly bounded functions such that $F^k_u (\cdot)$ converges to  $F(\cdot,u)$ $\mu_{\bar x}$-a.e.  
By dominated convergence $F^k_u (\cdot)$ converges to  $F(\cdot,u)$ also in $L^1(\mu_{\bar x})$, that is
\begin{equation}\label{eqn:approxFviaFk}
    \lim_{k \to \infty} \EVM{(\bar x, y)}{} {\int_0^{t^\star} \big | F_u^k(X_s, t^\star-s) - F(X_s, \psi(t^\star-s,u)) \big | ds } =0\,.
\end{equation}
By going to a subsequence if necessary we thus get the pointwise convergence~\eqref{eqn:F-approximation}.

\paragraph*{Step 2}Let $\{\rho^k\}_{k \in \N}$ be a sequence of  functions $\rho^k \in  \Ckc{\infty}{\stateX}$, $0 \leq \rho^k \leq 1$ increasing pointwise to $\indA{\stateX}$. 
Consider the Cauchy problem
\begin{equation}
    \label{eqn:cauchy_Fourier}
    \begin{cases}
    &\partial_t {\varphi_u^{k}}(t, x) = \genX {\varphi_u^{k}}(t, x) + {\varphi_u^{k}}(t,x) F^k_u(x, t), \quad t \in [0, t^\star], \, x \in \stateX, \, \\ 
    &\varphi^{k}_u(0, x) = \rho^k(x).
    \end{cases} \tag{$\text{CP}^{k}_{T, u}$}
\end{equation}
Define  a  function 
\[
\bm{f}^k \colon [0, t^\star] \times \Cinf{\stateX} \to \Cinf{\stateX} \colon (t,\bm{\varphi}) \mapsto \bm{\varphi}(\cdot) F^k_u(\cdot,t),
\]
so that \eqref{eqn:cauchy_Fourier} can be written in the form $\partial_t \varphi^k_u(t) = \genX \varphi^k_u(t) + \bm{f}^k (t, \varphi^k_u(t)) $. 
Since $F_u^k \in  \Ckb{1}{\stateX \times[0,t^\star]}$, the mapping $\bm{f}^k$ is continuously Fréchet-differentiable on $ [0, t^\star] \times \Cinf{\stateX}$. 
It follows from \cite[Theorem~6.1.5]{Pazy1983SemigroupsEquations} that there exists a classical solution $\varphi^k_u(t, \cdot)$ to~\eqref{eqn:cauchy_Fourier}.
Moreover, by the Feynman-Kac Theorem (see e.g.\ \cite[Theorem~17.4.10]{cohen2015stochastic}) we have the representation
\[
\varphi^k_u(t, x) = \EVM{(x,y)}{}{\rho^k(X_t) \exp\left(\int_0^t F^k_u(X_s,t^\star - s)ds\right) }, \quad (t, x) \in [0, t^\star] \times \stateX,
\]
where the value of the expectation is invariant with respect to $y$.
We get from \eqref{eqn:F-approximation}, bounded convergence  and the uniform boundedness of $F_u^k$ and $ \rho^k$  that for any $(t, x) \in [0, t^\star] \times \stateX$,
\begin{equation}\label{eq:limit-phiK}
 \lim_{k \to \infty} \varphi^k_u(t, x) = \varphi(t, x; u)\,.
\end{equation}
We now define for each $k \in \N$ the functions $G_u^k(t,x,y) = \varphi_u^k(t,x) e^{\innerp{\psi(t, u)}{y}}$ for $(t, x, y) \in [0, t^\star] \times \stateXY$. 
Note that for any $u \in \U$ it holds that $\psi(t, u) \in \U$ for all $t \geq 0$ (cf.\ \cite[Proposition~6.4]{Duffie2003AffineFinance}), and so the function $[0,t^\star] \times \stateY \ni (t, y) \mapsto e^{\innerp{\psi(t, u)}{y}}$ is bounded.
Thus, we have $L^1$-convergence of $G^{k}_u(t^\star - t, X_t, Y_t)$ to $M^\star_t := \varphi(t^\star - t, X_t; u) e^{\innerp{\psi(t^\star - t, u)}{Y_t}}$ as $k \to \infty$ for any $0 \leq t \leq t^\star$.
It remains to show that the limiting process $(M^\star_t)_{t\in[0, t^\star]}$ is in fact a complex-valued martingale.
A similar computation as in~\eqref{eqn:PDEcombGu} gives that
\begin{align*}
(\partial_t  + \weakgenX + &\genY) G_u^k(t^\star - t, x, y) \\ &= G_u^k(t^\star - t, x, y) \left ( F(x,\psi(t^\star-t,u)) - F_u^k(x,t^\star-t)\right ).
\end{align*}
Using Corollary~\ref{cor:afterMPresult}, we have that the process
\begin{align*}
M^k_t := G^k_u(&t^\star - t, X_t, Y_t) \\ &- \int_0^t G^k_u(t^\star - s, X_s, Y_s)\left(F(X_s, \psi(t^\star - s, u)) - F^k_u(X_s, t^\star - s) \right) ds,
\end{align*}
with $t \in [0, t^\star]$, is a martingale.
Moreover, since $G^k_u$ is bounded there is a constant $C$ such that
\begin{align}
    &\EVM{z}{}{\left|\int_0^t G^k_u(t^\star - s, X_s, Y_s)\left(F(X_s, \psi(t^\star - s, u)) - F^k_u(X_s, t^\star - s) \right) ds\right|}  \notag\\
    \label{eqn:Fubini_aux}
    &\qquad \leq C \, \EVM{z}{}{\left| \int_0^t \left(F(X_s, \psi(t^\star - s, u)) - F^k_u(X_s, t^\star - s) \right) ds \right|},
\end{align}
which converges to zero as $k \to \infty$ by~\eqref{eqn:approxFviaFk}.
Hence, the $L^1$-limit of $M^k_t$ as $k \to \infty$ is in fact equal to $M^\star_t$ for any $t \in [0, t^\star]$, from which we deduce that $(M^\star_t)_{t\in[0, t^\star]}$ is a martingale.
Finally, taking expectation leads to
\[\varphi(t^\star, x; u) e^{\innerp{\psi(t^\star, u)}{y}} = M_0^\star = \EVM{z}{}{M_{t^\star}^\star} = \EVM{z}{}{e^{\innerp{u}{Y_{t^\star}}}} .\]
\end{proof}

It is shown in~\cite[Section~4.4]{bib:ethier-kurtz-86} (in particular Theorem~4.4.2), that if any two solutions for the martingale problem associated with $(\gen, \domgen\gen)$ generate the same one-dimensional distributions of $(X, Y)$, then the martingale problem is well-posed and $(X,Y)$ has the strong Markov property. 
In fact, the one-dimensional distributions of  $(X,Y)$ can be characterized by similar arguments as in  Theorem~\ref{prop:fourier}, which  gives uniqueness for the martingale problem and hence establishes  existence of  MMAPs.
Given their importance for our purposes, we summarize these results in the following

\begin{theorem}\label{thm:existence}
Consider the linear operator $(\gen, \domgen{\gen})$ given by~\eqref{eqn:defGenXY}, where the underlying parameters $(a, \alpha, b, \beta, c, \gamma, m, \mu)$ are $x$-admissible and where $\genX$ satisfies Assumption~\ref{assumption:genX}.
Then the martingale problem for $(\gen, \allowbreak \domgen{\gen})$ is well-posed and there exists a Markov-modulated affine process $(X, Y)$ corresponding to $(\gen, \allowbreak \domgen{\gen})$.
Moreover, $(X, Y)$ has the strong Markov property.
\end{theorem}
\begin{proof}
Let $\Pp^1_z, \Pp^2_z \in \MP{\gen, \domgen{\gen}}$ for any $z \in \stateXY$.
Fix arbitrary $u \in i\R^n$, $f \in \domgen{\genX}$, $t \geq 0$, and define for $j \in \{1, 2\}$,
\begin{equation}
\varphi_u^{j,f}(t,x) = \EVM{(x,y)}{j}{\exp\left(\int_0^t F(X_s, \psi(t - s, u)) ds\right)f(X_t)}, \quad x \in \stateX.
\end{equation}
Similar to the proof of Theorem~\ref{prop:fourier}, it is then easy to deduce that 
\[\EVM{z}{j}{f(X_t)e^{\innerp{u}{Y_t}}} = \varphi^{j,f}(t,x) e^{\innerp{\psi(t, u)}{y}}, \quad j \in \{1, 2\}.\]
Further, since $\varphi^{j,f}_u$ only involves the law of $X$, we have that $\varphi_u^{1,f}(t,x) = \varphi_u^{2,f}(t,x)$.
Since $u \in i\R^n$, $t \geq 0$  and $f \in \domgen{\genX}$ were arbitrary, we conclude that the one-dimensional distributions of $(X, Y)$ with fixed starting point $z$ are the same under any solution to $\MP{\gen, \allowbreak \domgen{\gen}}$ and with Theorem~4.4.2.a of~\cite{bib:ethier-kurtz-86} we arrive at the well-posedness of the martingale problem.
The strong Markov-property of $(X,Y)$ under $(\Pp_z)_{z \in \stateXY} \subset \MP{\gen, \domgen{\gen}}$ then follows from Theorem~4.4.2.c of~\cite{bib:ethier-kurtz-86}.
\end{proof}

The next result gives conditions ensuring that $\varphi(\cdot; u)$ is in fact the solution of a suitable Cauchy problem.

\begin{proposition}\label{prop:fourier2}
Let $\varphi(\cdot ; u)$ be the function studied in Theorem~\ref{prop:fourier} and suppose that 
\begin{enumerate}[label=(\roman*)]
\item\label{item:fourier1} for any $x\in \stateX, \, \varphi(\cdot, x;u) \in \Ck{1}{[0,T]}$,
\item\label{item:fourier2} for any $t \in [0, T], \,  F(\cdot, \psi(t, u)), \, \varphi(t, \cdot;u) F(\cdot, \psi(t, u)) \in \Cb{\stateX}$.
\end{enumerate}
Then $\varphi(\cdot;u)$ solves Cauchy problem~\eqref{eqn:cauchy_prime} with $\gengenX = \weakgenX$ and initial condition $f \equiv 1$.
\end{proposition}
\begin{proof}
We fix an arbitrary $t \in [0, T)$ and consider some small $s \in (0, T)$.
Straightforward computations show that
\begin{align*}
    &\frac{\semiX_s\varphi(t, \cdot; u)(x) - \varphi(t, x; u)}{s} \\&\qquad = \frac{\varphi(t + s, x; u) - \varphi(t, x; u)}{s} \\
    &\qquad \quad+ \frac{1}{s}\EVM{(x, y)}{}{e^{\int^{t +s}_0F(X_r, \psi(t + s - r, u))dr} \left( e^{-\int^s_0F(X_r, \psi(t + s - r, u))dr} -1 \right)},
\end{align*}
where the first term converges by assumption~\ref{item:fourier1} to $\partial_t \varphi(t, x;u)$ as $s \to 0+$.
Regarding the second term, we yet again consider the limit as $s \to 0+$, and using dominated convergence along with the fact that $X$ has càdlàg paths said limit is equal to
\begin{align*}
   &\EVM{(x, y)}{}{\lim_{s \to 0+}e^{\int^{t +s}_0F(X_r, \psi(t + s - r, u))dr} \frac{1}{s}\left( e^{-\int^s_0F(X_r, \psi(t + s - r, u))dr} -1 \right)} 
   \\ &\qquad = -\EVM{(x, y)}{}{e^{\int^{t}_0F(X_r, \psi(t - r, u))dr} F(x, \psi(t, u))} 
   \\ &\qquad = -{\varphi}(t,x;u) F(x, \psi(t, u)),
\end{align*}
and so we get by condition~\ref{item:fourier2} that $\varphi(t, \cdot;u) \in \domgen{\weakgenX}$. 
\end{proof}

Solving the Cauchy problem is of course substantially simplified in specific applications with narrower model assumptions.
For example, if $X$ is a diffusion, then Problem~\eqref{eqn:cauchy_prime} reduces to a PDE.
In such a case, conditions for the existence of a classical solution to the above Cauchy problem are available in various degrees of generality, but are usually similar in character.
The coefficients of $\genX$ typically need to be locally Lipschitz, grow at most linearly and suffice a local ellipticity assumption.
Regarding the perturbation controlled by $F$, common assumptions translate to the boundedness of $\stateX \times [0,T] \ni (x,t) \mapsto F(x,\psi(t,u))$, and to the condition that $F(x,\psi(t,u))$ is locally Lipschitz (H\"older) in $x$ ($t$).
See e.g.\ \cite[Chapter~V]{bib:lsu-68} for further details.
Some classical references on more general Cauchy problems are \cite{Pazy1983SemigroupsEquations}, \cite{Thieme1990SemiflowsOperators}, \cite{Engel2000One-ParameterEquations} or \cite{Vrabie2003C0-SemigroupsApplications}.

%% file: 5.tex
\section{Further properties}\label{sec:further-properties}

In this section we study further properties of MMAPs. In Section~\ref{sec:Feller} we give conditions ensuring  the Feller property of strongly regular Markov-modulated affine semigroups.
Sections~\ref{sec:semimartingales} and~\ref{sec:expmoments} discusse the semimartingale property of MMAPs and the existence of real exponential moments, respectively.

\subsection{Feller Property}\label{sec:Feller}

Using our previous results we now show  the Feller property of strongly regular MMAPs by probabilistic arguments.
Note that strongly $x$-admissible parameters are a necessary condition for a process to be Feller, since the Feller property requires that $\gen(\domgen{\gen}) \subset \Cinf{\stateXY}$.

The standard approach for studying the existence of Feller processes is to apply the theory of Hille-Yosida. 
Since the operator $\genX$ may be unbounded, this inherently analytical task is challenging in our setup.
Detailed discussions and results on the generator of Feller processes can be found in e.g.\ \cite{Schilling1998FellerPaths}, \cite{Schilling1998ConservativenessSemigroups}, \cite{Schilling1998GrowthProcesses}, \cite[Chapter~VII]{revuz1999continuous} or \cite{Kuhn2018OnProcesses}.

\begin{proposition}\label{prop:feller}
Suppose that the parameters $(a, \alpha, b, \beta, \allowbreak c, \gamma, m, \mu)$ underlying $(\gen, \domgen{\gen})$ are strongly $x$-admissible,
 and that for any $x \in \stateX$,
\begin{equation}\label{eqn:momcondFeller}
    \int_{\stateY\backslash Q_0} \norm{\zeta}m(x, d\zeta) < \infty,
\end{equation}
where $Q_0 = \{\zeta \in \stateY \mid |\zeta_i| \leq 1, \, 1 \leq i \leq n\} \backslash \{0\}$.
Then the Markov-modulated affine process $(X,Y)$ corresponding to $(\gen, \domgen{\gen})$ is Feller.

Further, let $(L, \allowbreak \domgen{L})$ be the infinitesimal generator of the associated Feller semigroup. 
Then $\domgen{\gen}$ is a core of $L$.
\end{proposition}
\begin{proof}
Theorem~\ref{thm:existence} gives the existence of a sub-Markov semigroup $(P_t)_{t \geq 0}$ corresponding to $(X,Y)$.
To show the strong continuity of $(P_t)_{t\geq 0}$, we rely on the classical result~\cite[Theorem~III.2.4]{revuz1999continuous}, according to which the property that $\lim_{t \to 0+} P_t f(z) = f(z), \, \forall z \in \mathcal{D}$ and any $f \in \Cinf{\mathcal{D}}$ is sufficient for strong continuity.
We take up the definition of the set $\overline{\Theta}_0$ from the proof of Lemma~\ref{lemma:domdenseCinf} and pick an arbitrary $\tilde{f} = fh \in \overline{\Theta}_0$.
We further fix arbitrary $T > 0$ and $u \in \U$.
Introduce the function $\bm{f}$ via 
\[\bm{f} : [0, T] \times \Cinf{\stateX} \to \Cinf{\stateX} : (t,\bm{\varphi}) \mapsto \bm{\varphi} F(\cdot, \psi(t, u)),
\]
which is well-defined since $F(\cdot, \psi(t, u)) \in \Cb{\stateX}$.
The moment condition~\eqref{eqn:momcondFeller} ensures that $F(x, \cdot) \in \Ck{1}{\U}$ (cf.\ \cite[Lemma~5.3.ii]{Duffie2003AffineFinance}), so that $\bm{f}$ is $C^1$ in $t$.
Further, the function $\bm{f}$ is continuously Fréchet-differentiable in its second argument.
Thus, by Theorem~6.1.5 of \cite{Pazy1983SemigroupsEquations} there exists a unique classical solution to Problem~\eqref{eqn:cauchy_prime}.
Similar to the proof of Theorem~\ref{prop:fourier}, we then introduce the process
\[
G_u^f(t - s, X_s, Y_s) := \varphi^f(t - s, X_s; u) e^{\innerp{\psi(t - s, u)}{Y_s}}, \quad s \in [0, T],
\]
where $\varphi^f(\cdot, x;u)$ solves~\eqref{eqn:cauchy_prime} for $\gengenX = \genX$ with $\varphi^f(0, x; u) = f(x)$, for any $x \in \stateX$.
By a similar argument as in the proof of Theorem~\ref{prop:fourier}, the process $G^f_u(t - s, X_s, Y_s), \, s \in [0, T],$ is a martingale.
So, for any $(x, y) \in \stateXY$, and appropriate $g \in \Ckc{\infty}{\R^{n-m}}$, $v \in \C^m_{--}$ (compare with the definition of $\overline{\Theta}_0$), it holds that
\begin{align*}
    P_t \tilde{f}(x, y) &= \EVM{z}{}{f(X_t) \int_{\R^{n-m}} e^{\innerp{(v, iq)}{(Y_t^+, Y_t^*)}} g(q) dq} \\
    &= \int_{\R^{n-m}}\EVM{z}{}{G^f_{(v, iq)}(0, X_t, Y_t) }dq \\
    &= \int_{\R^{n-m}} \varphi^f(t, x; v, iq) e^{\innerp{\psi(t, v, iq)}{y}} g(q) dq,
\end{align*}
where the second equality follows from Fubini's theorem.
Recall the definition of functions $F$ and $R$ in~\eqref{eqn:Riccati_F}--\eqref{eqn:Riccati_RJ}.
Via an application of dominated convergence, we then obtain
\begin{align*}
    \partial_t^+ P_t \tilde{f}(x, y)|_{t = 0} &= \int_{\R^{n-m}} \partial_t^+ \left\{\varphi^f(t, x; v, iq) e^{\innerp{\psi(t, v, iq)}{y}} g(q) \right\}|_{t = 0} dq \\
    &= \int_{\R^{n-m}} \left(\genX \varphi^f(0, x; v, iq) + \varphi^f(0, x; v, iq)F(x, \psi(0, v, iq)) \right)e^{\innerp{\psi(0, v, iq)}{y}} \\
    & \qquad + \innerp{R(\psi(0, v, iq))}{y}e^{\innerp{\psi(0, v, iq)}{y}}\varphi^f(0, x; v, iq) g(q) dg \\
    &= \int_{\R^{n-m}} (\genX f(x) + F(x, v, iq) + \innerp{R(v, iq)}{y}f(x))e^{\innerp{(v, iq)}{y}}g(q) dq\\
    &= \gen \tilde{f}(x, y),
\end{align*}
which implies $\bplim{t \to 0+} P_t \tilde{f} = \tilde{f}$.
The remaining arguments are identical to the ones from the proofs of \cite[Proposition~8.2]{Duffie2003AffineFinance} and \cite[Proposition~6.3]{Filipovic2005Time-inhomogeneousProcesses} by noting that the linear hull of $\overline{\Theta}_0$ is dense in $\Cinf{\stateXY}$, and so we conclude that $(P_t)_{t \geq 0}$ is a Feller semigroup.
Finally, the well-posedness of $\MP{\gen, \domgen{\gen}}$ (see Theorem~\ref{thm:existence}) implies that $\domgen{\gen}$ is a core of $L$, cf.\ \cite{Okitaloshima1996OnProblem}.
\end{proof}
\begin{remark}
We use the moment condition~\eqref{eqn:momcondFeller} in the above proposition to deduce that there exists a classical solution to~\eqref{eqn:cauchy_prime}.
Accordingly, we could have instead assumed the solvability of the Cauchy problem to arrive at the Feller property, but we believe the more explicit condition~\eqref{eqn:momcondFeller} to be of greater use for applications.
\end{remark}

\subsection{Semimartingale Property}\label{sec:semimartingales}

MMAPs can in general explode and get killed, both of which prompt a transition to the cemetery state $\{\Delta\}$. To formalize these concepts we introduce various  stopping times; our definitions follow \cite{Cheridito2005EquivalentProcesses}. First, we  define the lifetime of $Y$ as
\[
T_\Delta := \inf\{t \geq 0 \mid Y_t = \Delta\}.
\]
Second, we let $\norm{\Delta} = \infty$ and we introduce the following sequences of stopping times,
\begin{align*}
    T_k' &:= \inf \{t \geq 0 \mid \norm{Y_{t-}} \geq k \text{ or } \norm{Y_t} \geq k \}, \quad k \in \N, \\
    T_k &:=
    \begin{cases}
       T_k', \quad & \text{if } T_k' < T_\Delta, \\
     \infty, \quad &  \text{if } T_k' = T_\Delta,
    \end{cases} \quad k \in \N.
\end{align*}
Moreover, we define the  \emph{explosion time}  of $Y$ as $T_\infty := \lim_{k \to \infty} T_k$. Then the process $Y$ explodes if and only if $T_\infty < \infty$;  killing corresponds to the event $T_\Delta < T_\infty$ and a trajectory of $Y$ belongs to $\stateY$ for all $t \ge 0$ if and only if $T_\Delta = T_\infty = \infty$.
By construction, as $(T_k \wedge k)$ announces $T_\infty$, the explosion time is predictable, which allows us to stop $Y$ at a time before $T_\infty$ and to work on $[0, T_\infty)$ for studying the semimartingale property.

\begin{proposition}\label{prop:semimart}
Let $(X, Y)$ be a Markov-modulated affine process with  $x$-admissible parameters $(a, \alpha, b, \beta, c, \gamma, m, \mu)$, where  $c(\cdot) \equiv 0$ and $\gamma = 0$.
Then $\Pp_z(0 < T_\Delta < T_\infty) = 0$ (no killing) and  the process $Y$ is a $\stateY$-valued semimartingale on $[0, T_\infty)$.
Moreover,  $Y$ admits the $\{\calF_t\}$-characteristics $(B, C, \nu)$ with respect to the truncation function  $\chi$.
Here
\begin{align}
    \label{eqn:semimart_drift}
    B_t &= \int_0^t \indA{\{t< T_{\infty}\}}(\tilde{b(X_{s})} + \tilde{\beta} Y_s)ds, \\
    \label{eqn:semimart_diff}
    C_t &= 2 \int_0^t \indA{\{t< T_{\infty}\}} \left(a(X_s) + \sum_{i = 1}^m \alpha_i Y_s^{+,i} \right) ds, \\
    \label{eqn:semimart_jump}
    \nu(dt, d\xi) &= \indA{\{t< T_{\infty}\}} \left(m(X_t, d\xi) + \sum_{i = 1}^m Y_t^{+, i} \mu_i(d \xi) \right) dt,
\end{align}
where the function $  \tilde{b} : \stateX \to \stateY $ is given by $\tilde b(x) =  b(x) + \int_{\stateY \backslash \{0\}} (\chi_\indexI (\xi), 0) m(x, d\xi)$,  and where
\begin{align*}
    \tilde{\beta}_{kl} &:= \begin{cases} \beta_{kl} + (1 - \delta_{kl}) \int_{\stateY \backslash \{0\}} \chi_k (\xi) \mu_l (d \xi), \quad &\text{if } l \in \indexI, \\
    \beta_{kl}, &\text{if } l \in \indexJ, \, \text{for } 1 \leq k \leq n.\end{cases}
\end{align*}
\end{proposition}
\begin{proof}
At first we show that  $\Pp_z(0 < T_\Delta < T_\infty) = 0$.
We use a similar argument as in the proof of~\cite[Lemma~3.1]{Cheridito2005EquivalentProcesses}.
Take a sequence $(g_l) \in \Ckc{\infty}{\stateY}$ with $0 \leq g_l \leq 1$ and $g_l = 1$ on $D_l := \{y \in \stateY \mid \norm{y} < l \}$.
Note that $g_l(y) = 1$ for $l$ large enough and that $\lim_{l \to \infty} g_l(Y_{t \wedge T_k}) = \ind{t \wedge T_k < T_\Delta}\,\Pp_z\text{-a.s.}$.
For each $l$ we obviously have that $g_l \in \domgen{\gen}$, where we tacitly think of $(x, y) \mapsto g_l(y)$ as a constant function in $x$.
Since $X$ is conservative, we have $\weakgenX 1 = 0$ so that the process given by
\[
M_t^{l}  := g_l(Y_t) - g_l(y) - \int_0^t \genY g_l(Y_s) ds, \quad t \geq 0,
\]
is a martingale.
Take $l > k$ and recall that $c \equiv 0$ and $\gamma = 0$.
It then holds that
\begin{align*}
    M^l_{t \wedge T_k} &= g_l(Y_{t \wedge T_k}) - g_l(y) - \int_0^{t \wedge T_k} \Big( \int_{\stateY\setminus D_{l-k}} (g_l(Y_s + \zeta) - 1)m(X_s, d \zeta) \\
    &\quad + \sum_{i = 1}^m\int_{\stateY\setminus D_{l-k}} (g_l(Y_s + \zeta) - 1) Y^i_s \mu_i(d \zeta)\Big) ds.
\end{align*}
Note that for any $i \in \indexI$ it holds that
\begin{equation}
    \label{eqn:proofSemimart_ineq}
\int_0^{t \wedge T_k} \int_{\stateY\setminus D_{l-k}}|g_l(Y_s + \zeta) - 1| Y^i_s \mu_i(d \zeta) ds \leq k \int_0^{t \wedge T_k} \mu(\stateY\setminus D_{l-k}) ds,
\end{equation}
which converges almost surely to $0$ as $l \to \infty$ by Lebegue's dominated convergence theorem.
A similar argument also applies with $m(X_s, \cdot)$.
Moreover, for fixed $k \leq l$, with inequalities of type~\eqref{eqn:proofSemimart_ineq} it is easily seen that almost surely
\[
|M^l_{t \wedge T_k}| \leq 1 + c_k t,
\]
for some constant $c_k$.
Thus, $-M^l_{t \wedge T_k}$ converges in $L^1$ to $1 - \ind{t \wedge T_k < T_\Delta} = \ind{0 < T_\Delta \leq t \wedge T_k}$ as $l \to \infty$ for all $t \geq 0$, which shows that $\ind{0 < T_\Delta \leq \cdot \wedge T_k}$ is a martingale.
However, as it is non-increasing it needs to be almost surely constant, which implies $\Pp_z(0 < T_\Delta < T_\infty) = 0$.

Next, for $u \in \R^n$ we introduce the process
\[
A(u)_t := i\innerp{u}{B_t} - \frac{1}{2}\innerp{u}{C_t u} + \int \left(e^{i \innerp{u}{y}} - 1 - i \innerp{u}{\chi(y)} \right) \nu([0,t] \times dy), \quad t \in [0, T_\infty),
\]
where $B, C$ and $\nu$ are as in~\eqref{eqn:semimart_drift}--\eqref{eqn:semimart_jump}.
Note that $A(u)_t = \int_0^t F(X_s, u) + \innerp{R(u)}{Y_s}ds$, where $F$ and $R$ are specified in~\eqref{eqn:Riccati_F} and \eqref{eqn:Riccati_Ri}--\eqref{eqn:Riccati_RJ}, respectively.
The process $e^{i \innerp{u}{Y}} - \int_0^\cdot \mathfrak{L}^{Y | X}(X_s)  e^{i \innerp{u}{Y_s}} ds$ is a martingale on $[0, T_\infty)$ by Corollary~\ref{cor:afterMPresult}.
And, for $t \in [0, T_\infty)$
\begin{align*}
    e^{i \innerp{u}{Y_t}} - \int_0^t \mathfrak{L}^{Y | X}(X_s)  e^{i \innerp{u}{Y_s}} ds &= e^{i \innerp{u}{Y_t}} - \int_0^t e^{i \innerp{u}{Y_s}} \left(F(X_s, u) + \innerp{R(u)}{Y_s} \right) ds \\
    &= e^{i \innerp{u}{Y_t}} - \int_0^t e^{i \innerp{u}{Y_s}} dA(u)_s,
\end{align*}
and the result follows from Theorem~II.2.42 of \cite{bib:jacod-shiryaev-03}.
\end{proof}

We continue with a few implications of Proposition~\ref{prop:semimart}. 
First, take the characteristics $(B, C, \nu)$ from Proposition~\ref{prop:semimart}, set $J_t(\chi) := \sum_{s \leq t} (\Delta Y_s - \chi(\Delta Y_s))$ and denote the integer-valued random measure counting the jumps of $Y$ by $\pi(ds, d\zeta)$.
Then, we obtain the following canonical decomposition of $Y$,
\begin{equation}
    \label{eqn:semimartdecomp_Z}
    Y = Y_0 + Y^c + \int_0^\cdot\int_{\stateY \backslash \{0\}} \chi(\zeta) \left(\pi(ds, d\zeta) - \nu(ds, d\zeta) \right) + J(\chi) + B,
\end{equation}
where $Y^c$ denotes the continuous local martingale part of $Y$, so that $\langle Y^{i,c}, Y^{j,c} \rangle = C^{ij}$ for $i, j \in \indexI \cup \indexJ$.
Second, as a conservative Feller process with càdlàg sample paths, $X$ is a semimartingale (cf.\ Lemma~3.2 of \cite{Schilling1998GrowthProcesses}).
Hence  under the assumptions of Proposition~\ref{prop:semimart}, we have that $Z = (X,Y)$ is a semimartingale on $[0, T_\infty)$.

\subsection{Exponential moments}\label{sec:expmoments}

In this section we resort to classical semimartingale methods in order to study real exponential moments of $Y_t$.
For this we need a variant of \eqref{eqn:cauchy_prime}, namely
\begin{equation}
    \label{eqn:cauchy_expmoments}
    \begin{cases}
    &\partial_t {\widetilde{\varphi}}(t, x; u) = \weakgenX {\widetilde{\varphi}}(t, x; u) + {\widetilde{\varphi}}(t,x;u) F(x, \widetilde\psi(t, u)), \quad t \in [0, T], \, x \in \stateX, \\
    & {\widetilde{\varphi}}(0, \cdot; u) \equiv 1,
    \end{cases} \tag{$\text{CP}^\star_{T, u}$}
\end{equation}
with $T > 0$ and where $F$ is as in Theorem~\ref{prop:fourier}, albeit with the modification that $u \in \R^d$, and where $\widetilde\psi$ is analogously a modification of $\psi$, that is $\widetilde\psi : [0, T] \times \R^n \to R^n$ solves
\begin{align}
    \label{eqn:riccati_real_u}
    \partial_t \widetilde\psi (t, u) &= R (\widetilde\psi (t, u)) , \quad \widetilde\psi (0, u) = u.
\end{align}
We define a classical solution of~\eqref{eqn:cauchy_expmoments} analogously to the previous case of~\eqref{eqn:cauchy_prime}. 

\begin{proposition}\label{prop:expmoments_finite}
Let $z \in \stateXY$, $T \geq t^\star \geq 0$ and $u \in \R^n$ be arbitrary and consider $\Pp_z \in \MP{\gen, \allowbreak \domgen{\gen}}$.
Further, let $(X,Y)$ be a Markov-modulated affine process, where in addition $Y$ is a $\stateY$-valued semimartingale with $\{\calF_t\}$-characteristics $(B, C, \nu)$ given by~\eqref{eqn:semimart_drift}--\eqref{eqn:semimart_jump}.
Suppose that $\widetilde\psi(\cdot, u)$ is $C^1$ and solves~\eqref{eqn:riccati} and that the function $[0,T] \to \Cb{\stateX} : t \mapsto \widetilde{\varphi}(t, \cdot; u)$ is a classical solution of~\eqref{eqn:cauchy_expmoments}.
Suppose the parameters $(a, \alpha, b, \beta, c, \gamma, m, \mu)$ are strongly $x$-admissible and further satisfy the following conditions,
\begin{enumerate}[label=(\roman*)]
\item $c(\cdot) \equiv 0$ and $\gamma = 0$,
\item $\int_{\{|\zeta| > 1\}} e^{\innerp{\psi(t, u)}{\zeta}}{m(x, d\zeta)} < \infty, \, \forall (t, x) \in [0, T] \times \stateX$,
\item $\int_{\{\zeta_k > 1\}}\zeta_k e^{\innerp{\psi(t, u)}{\zeta}} \mu_l(d \zeta) < \infty, \, 1 \leq k, l \leq m, \forall t \in [0, T]$.
\end{enumerate}
Then,
\[
\EVM{z}{}{e^{\innerp{u}{Y_{t^\star}}}} = \widetilde\varphi({t^\star}, x; u) e^{\innerp{\widetilde\psi({t^\star}, u)}{y}}.
\]
\end{proposition}
\begin{proof}
Fix an arbitrary $t^\star \in [0, T]$ and use $\calF^X_{t^\star}$ to denote the $\Pp_z$-completion of $\sigma(X_s : s \in [0, t^\star])$.
Conditional on $\calF^X_{t^\star}$, $Y_{t^\star}$ is distributed as a time-inhomogeneous affine semimartingale evaluated at $t^\star$, the exponential moments of which are handled in~\cite{Kallsen2010ExponentiallyProcesses}.
Within our setup, we readily fulfill the conditions 1, 2, 4 and 5 of ~\cite[Theorem~5.1]{Kallsen2010ExponentiallyProcesses}.
Regarding condition 3, we introduce the function $\varphi^0 : [0,T] \times \R^d \times D_{\stateX} \to \R$ via
\[
\varphi^0(t, u, f) = \int_0^t F(f(s), \widetilde\psi(t - s, u)) ds.
\]
Then, an application of \cite[Theorem~5.1]{Kallsen2010ExponentiallyProcesses} gives that
\[
\EVM{z}{}{e^{\innerp{u}{Y_{t^\star}}}\mid \calF^X_{t^\star}} = e^{\varphi^0({t^\star}, u, X) + \innerp{\widetilde\psi({t^\star}, u)}{y}}.
\]
We need to show that $\EVM{z}{}{e^{\varphi^0({t^\star}, u, X)}} = \widetilde{\varphi}({t^\star}, x; u)$ in order to conclude the proof.
To do so, note that by assumption $\widetilde{\varphi}$ solves~\eqref{eqn:cauchy_expmoments} and by an application of Corollary~\ref{cor:afterMPresult} we get the martingality of
\[
 M_t^{\widetilde{\varphi}} := \widetilde\varphi(t^\star - t, X_t ; u) + \int_0^t \widetilde\varphi(t^\star - s, X_s ; u) F(X_s, \widetilde\psi(t^\star - s, u)) ds, \quad t \in [0, T].
\]
Integration by parts and collecting terms gives us
\begin{align}
    \widetilde{\varphi}(t^\star - t, X_t\calF\,;\, &u) \exp\left(\int_0^t F(X_s, \widetilde\psi(t^\star - s, u)) ds\right) \notag\\
    \label{eqn:expmoments_real_proof2}
    & = \int_0^t \exp\left(\int_0^s F(X_r, \widetilde\psi(t^\star - r))dr \right) dM_s^{\widetilde{\varphi}}.
\end{align}
The local martingale in~\eqref{eqn:expmoments_real_proof2} is in fact a true martingale since the integrand is bounded and since $M^{\widetilde{\varphi}}$ is a square integrable martingale (cf.\ \cite[Theorem~IV.11]{protter2005stochasticb}).
So, by rearranging and taking expectations we arrive at the desired result.
\end{proof}
The general results of the appendix of \cite{Duffie2003AffineFinance}  provide further conditions for the characterization of exponential moments of $Y_t$.
To begin with, let $U \subset \C^n$ be an open neighbourhood of $0$.
If $\varphi(t, \cdot)$ and $\psi(t, \cdot)$ have analytic extensions on $U \times \stateX$ and $U$, respectively, then the exponential moments $\EVM{(x,y)}{}{e^{\innerp{u}{Y_t}}}$ are finite for all $u \in U \cap \R^n, (x, y) \in \stateX \times \stateY, 0 < t \leq T$.
Moreover, under the same conditions it also holds that
\[
\EVM{(x,y)}{}{e^{\innerp{u}{Y_t}}} = \varphi(t, x; u) e^{\innerp{\psi(t, u)}{y}},
\]
for all $u \in U$ such that $\operatorname{Re}(u) \in U \cap \R^n, (x, y) \in \stateX \times \stateY$.
As the function $\psi$, or more precisely the system of generalized Riccati equations defining it, is the same as in the standard affine case, we can outsource the study of suitable extensions for $\psi$ by referencing existing results, most notably the systematic treatment of exponential moments of affine processes in \cite{Keller-Ressel2015ExponentialProcesses}.
The study of regularity properties of $u \mapsto \varphi(t,x;u)$ on the other hand is not straightforward considering the definition of $\varphi$ in Theorem~\ref{prop:fourier}.

%% file: 6.tex
\section{Applications in Mathematical Finance}\label{sec:applications}

In this section we discuss applications of MMAPs in mathematical finance. Our goal is to illustrate the wide range of modelling possibilities offered by this class of processes. We begin with an extension of Theorem~\ref{prop:fourier} that enables us to solve many pricing problems with Fourier inversion techniques. Moreover,  we consider in detail a model for the joint pricing of bonds,  equity options and credit derivatives that captures many stylized facts of financial data and we sketch several further applications.
Throughout this section we fix some $z \in \stateXY$ and assume that $(X,Y)$ is a Markov-modulated affine semimartingale under $\Pp_z$.

\subsection{Derivative pricing via Fourier inversion techniques} \label{subsec:fourier-pricing}

In  many financial applications one needs to evaluate expectations of the form
\begin{equation}
    \label{eqn:discounting}
Q_t f(z) := \EVM{z}{}{\exp\left(\int_0^t L(Y_s)ds\right)f(X_t, Y_t)}, \quad t \geq 0, \, f \in \Bb{\stateXY},
\end{equation}
for an affine transformation $L : \stateY \to \R$,  $y \mapsto  l + \innerp{\lambda}{y}$, for some $l \in \R$, $\lambda \in \R^n$. This can be done efficiently via Fourier inversion techniques. For this we need  to compute for all  $(q, u) \in i\R^d \times i \R^n $ the expectation
$Q_t f^{(q,u)} (z) $ where  $f^{(q,u)}(x,y) =  e^{\innerp{(q,u)}{(x,y)}}$.
As in the standard affine case, this can be achieved using simple modifications of the functions $F$ and $R$ (see~\eqref{eqn:Riccati_F}--\eqref{eqn:Riccati_RJ}). Replace $F(x,u)$ by $\widetilde{F}(x, u) = F(x, u) + l$ and $R(u)$ by $\widetilde{R}(u) = R(u) + \lambda$ in the equations of Theorem~\ref{prop:fourier}, fix some $t >0$,
 pick arbitrary $(u,q) \in \U \times i\R^d$ and suppose $\widetilde\varphi$ is a classical solution of
\begin{equation*}
    (\partial_s - \weakgenX) \widetilde\varphi(s, x, u) = \widetilde\varphi(s, x, u) \widetilde{F}(x, \widetilde\psi(s, u)), \quad \widetilde\varphi(0, x, u) = e^{\innerp{q}{x}}, \, \, \forall x \in \stateX,
\end{equation*}
with
\begin{align*}
    \partial_s \widetilde{\psi} (s, u) &= \widetilde{R} (\widetilde{\psi} (s, u)) , \quad \widetilde{\psi} (0, u) = u.
\end{align*}
Then, similar arguments as the ones preceding Theorem~\ref{prop:fourier} show that
\begin{equation}\label{eqn:extended_transform}
    \EVM{z}{}{\exp\left(\int_0^t L(Y_s)ds\right)e^{\innerp{(q, u)}{(X_t, Y_t)}}} = \widetilde\varphi(t, x, u) e^{\innerp{\widetilde{\psi}(t, u)}{y}}. 
\end{equation}
Suppose there exists analytic extensions of $\widetilde\varphi(t, x, \cdot)$ and $\widetilde\psi(t, \cdot)$ to some $u \in \C^n\backslash \U$, then the transform formula~\eqref{eqn:extended_transform} also holds for this particular $u$.

If $\Pp_z$ is a risk-neutral pricing measure, then expectations in the form of Equation~\eqref{eqn:discounting} are typically understood as the price of a financial instrument with payoff $f(X_t, Y_t)$ and with risk-free short rate $-L(Y_s)$.
Suppose $Y_t$ represents the value of some underlying and we are interested in the price of a derivative with payoff $h(Y_t)$, then it is common to find a representation of $h$ in the form of
\[
    h(y) = \int_{\R^q} e^{\innerp{u + iAv}{y}} \Tilde{h}(v) dv,
\]
for some matrix $A \in \R^{n \times q}$ with $q \leq n$, for an appropriate integration kernel $\Tilde{h} : \R^q \to \C$, and where $u$ is chosen such that $\EVM{z}{}{e^{\innerp{u}{Y_t}}} < \infty$.
The integration kernel $\Tilde{h}$ can be determined by Fourier-inverting $h$.
For many popular payoff functions, such inversions are explicitly available, see e.g.\ \cite[Section~10.3.1]{filipovic2009term} for univariate examples or \cite{Eberlein2010AnalysisApplications} (and the references therein) for multivariate payoffs.
In such a setup, $Q_t h $ reduces to
\begin{equation}
\label{eqn:int_kernerl_rep_payoff}
\int_{\R^q} \widetilde\varphi(t, x, u + iAv) e^{\innerp{\widetilde{\psi}(t, u + iAv)}{y}} \Tilde{h}(v) dv,
\end{equation}
which is easily computed using numerical integration.

\subsection{Joint Pricing of Equity Options and Credit Derivatives}

Next we show  how Markov-modulated affine diffusions can be used for the joint and dynamically consistent valuation of equity and credit derivatives written on the same underlying firm. More precisely, we consider a firm that may default at some doubly stochastic  default time $\tau$ (see e.g.\ \cite[Chapter~10.6]{bib:mcneil-frey-embrechts-15}) and we introduce a model for the joint dynamics of the risk-free short rate $r$, of the hazard rate $\gamma$ of $\tau$, of the pre-default logarithmic stock price $p$ and of the stock price volatility  $v$. In this section we view $\Pp_z$ as risk-neutral pricing measure.
One salient feature  of our model is that it allows for  negative association between the risk-free short rate and the firm's hazard rate as well as its volatility process.
Modelling negative dependence between \emph{positive} diffusions (i.e.\ processes of `CIR-type') is not feasible within the standard affine framework, see e.g.\ \cite{bib:duffie-liu-01}.
Here we overcome this well-known drawback of affine processes  rather naturally via the common dependence of the positive processes on the modulating  process $X$.

\paragraph*{The Model} We suppose that $X$ is a solution of  the following SDE,
\begin{equation}
\label{eqn:JacobiX}
     dX_t = -\kappa (X_t - \theta)dt + \sigma\sqrt{X_t(1 - X_t)}dW_t^X,
\end{equation}
for some standard Brownian motion $W^X$ and with parameters $\theta \in [0,1], \kappa, \sigma \geq 0$.
It follows that $X$ is a Jacobi process on the state space $\stateX = [0,1]$ with
\[
\genX f(x) = -\kappa(x - \theta)\frac{\partial f(x)}{\partial x}  + \frac{1}{2} \sigma^2 (x(1-x))\frac{\partial^2 f(x)}{\partial x^2}, \quad f \in \domgen{\genX} \subset \Ckc{2}{\stateX}.
\]

The risk-free short rate $r$ and the firm's hazard rate $\gamma$ are given as weak solutions of the following SDEs,
\begin{align}
    \label{eqn:creditCIR}
        dr_t &= \left(\bar{b}_r + b_r X_t + \beta_r r_t \right) dt + \sqrt{2\alpha_r r_t} dW_t^r, \\
    d\gamma_t &= \left(\bar{b}_\gamma + b_\gamma (1 - X_t) + \beta_\gamma \gamma_t \right) dt + \sqrt{2\alpha_\gamma \gamma_t} dW^\gamma_t,
\end{align}
for independent  one-dimensional Brownian motions $W^r$ and $W^\gamma$  and with constants $\alpha_r, \alpha_\gamma > 0$, $\bar{b}_r, b_r, \bar{b}_\gamma, b_\gamma \geq 0$ and  $\beta_\gamma, \beta_r \in \R$.

We consider the following extension of the Heston model for  the pre-default dynamics of the logarithmic stock price $p$ and its instantaneous variance $v$
\begin{align}
    \label{eqn:stockLogP}
    dp_t &= \left(r_t + \gamma_t - \frac{1}{2} v_t\right) dt + \sqrt{v_t} dW^p_t, \\
    \label{eqn:stockVar}
    dv_t &= \left(\bar{b}_v + b_v(1 - X_t) + \beta_v v_t \right) dt + \sqrt{2 \alpha_v v_t} dW_t^v,
\end{align}
where $W^p$ and $W^v$ are two Brownian motions with correlation $\rho \in [-1,1]$.
Moreover, we have constants $\alpha_v > 0$, $\bar{b}_v, b_v \geq 0$ and $\beta_v < 0$. We assume that at $\tau$ the stock price jumps to zero,
that is we assume that $S_t = \indA{\{\tau >t\}} e^{p_t}$. This gives the following stock price dynamics
\begin{equation} \label{eq:dSt}
dS_t =  r_t S_t dt  +  \sqrt{v_t} S_t dW^p_t - S_{t-} d M_t^\tau \quad \text{ with } M_t^\tau = \indA{\{\tau \le t\}} - \int_0^{\tau \wedge t} \gamma_s ds.
\end{equation}
Note that $S$ is a martingale which justifies the interpretation of $\Pp_z$ as risk neutral measure.
The exogenous process  $X$ can be viewed  as  state of the economy, with values close to $1$ ($0$) representing expansion (recession). 
Note that the drift of $r$ is increasing in $X_t$ while the drift  of $\gamma$ is decreasing in $X_t$. 
The opposite impact of $X$ on the drift of $\gamma$ and $r$ allows to model negative association between the short rate and the hazard rate. 
This phenomenon is empirically well documented, cf.\ \cite{Longstaff1995ADebt} or \cite{Duffee1998TheSpreads}; it might be due to the fact that in times of financial crisis where hazard rates are high, central banks typically lower the reference rates.
Moreover, the proposed specification creates natural dependence between the risk-neutral probability of default and the return distribution of $p$:  as the economic environment worsens ($X$ decreases), the risk of default increases, as does the volatility of the stock returns.
Stock option pricing with credit risk is for example treated in~\cite{Carr2006AProcesses} or~\cite{carr2009stock}.

\paragraph*{Mathematical aspects} Next we show that the dynamics \eqref{eqn:JacobiX}, \eqref{eqn:creditCIR}, \eqref{eqn:stockLogP}, \eqref{eqn:stockVar} fit into our framework of Markov-modulated affine processes.  First note that $X$ is a polynomial process (see~\cite{Cuchiero2012PolynomialFinance}) with compact support, and as such it is Feller (cf.\ Proposition~4.1 of \cite{timeinhompolynomial2020}).
The coefficients of $X$ satisfy  the regularity assumptions of Theorem~8.2.1 in~\cite{bib:ethier-kurtz-86}, so that $\Ckc{\infty}{\stateX}$ is a core of $(\genX, \domgen{\genX})$. Overall, Assumption~\ref{assumption:genX} is met. Next  we let $Y = (Y^1, Y^2, Y^3, Y^4)^\transpose = (r, \gamma, v, p)^\transpose$ so that    $\stateY = \R_+^3\times \R$  and $m = 3$, $n = 1$.  Set
\begin{align*}
    \alpha_1 &= \begin{pmatrix}
    \alpha_r & 0 & 0 & 0\\
    0 & 0 & 0 & 0 \\
    0 & 0 & 0 & 0 \\
    0 & 0 & 0 & 0
    \end{pmatrix}, \quad
    \alpha_2 = \begin{pmatrix}
    0 & 0 & 0 & 0\\
    0 &  \alpha_\gamma & 0 & 0 \\
    0 & 0 & 0 & 0 \\
    0 & 0 & 0 & 0
    \end{pmatrix}, \quad
    \alpha_3 = \begin{pmatrix}
    0 & 0 & 0 & 0\\
    0 & 0 & 0 & 0 \\
    0 & 0 & \alpha_v & \sqrt{\alpha_v} \rho \\
    0 & 0 & \sqrt{\alpha_v} \rho & 1/2
    \end{pmatrix}, \\
    b(x) &= \begin{pmatrix}
    \bar{b}_r + b_r x \\
    \bar{b}_\gamma + b_\gamma (1 - x) \\
    \bar{b}_v + b_v(1 - x) \\
    0
    \end{pmatrix}, \quad
    \beta = \begin{pmatrix}
    \beta_r & 0 & 0 & 0\\
    0 & \beta_\gamma & 0 & 0 \\
    0 & 0 & \beta_v & 0 \\
    1 & 1 & -1/2 & 0
    \end{pmatrix},
\end{align*}
and let $\alpha = (\alpha_1, \alpha_2, \alpha_3)$. With these definitions the dynamics of  $Y$ are governed by the  $x$-admissible parameters $(0, \alpha, b(x), \beta, 0, 0, 0, 0)$.

The Cauchy problem~\eqref{eqn:cauchy_prime} reduces to the following second order PDE,
\begin{align*}
    \partial_t \varphi(t, x; u) &= -\kappa(x - \theta)\partial_x \varphi(t, x; u) + \frac{1}{2} \sigma^2 (x(1-x))\partial_{xx} \varphi(t, x; u) \\
    &\qquad + \varphi(t, x; u)(\innerp{b(x)}{\psi(t, u)} + \innerp{a(x)\psi(t, u)}{\psi(t, u)}), \quad u \in \U, \\
    \varphi(0, x; u) &= f(x), \quad f \in \Ck{2}{[0,1]}.
\end{align*}
The existence of a classical solution to the above PDE follows from~\cite[Theorem~6.1.6]{Pazy1983SemigroupsEquations}. 

\paragraph*{Derivative pricing}
We work on an enlarged filtration $\calG_t = \calF_t \vee \sigma(\ind{\tau \leq s} : 0 \leq s \leq t), \,t \geq 0$ where $(\mathcal{F}_t )$ is the filtration generated by the processes $Z=(X,Y)$, so that $(\calG_t)$ contains information regarding the occurrence of default. Here we discuss the pricing of so-called \emph{survival claims}  with payoff $H_T = \ind{\tau > T} h(Y_T)$; simple examples would be a call option on the stock where $h(Y_T) = (e^{p_T}-K )^+$ or a defaultable bond with zero recovery where $h(Y_T) =1 $.  Recall that $\Pp$ is the risk neutral measure. Hence the price of a survival claim is given by
$$
\EVM{z}{}{\ind{\tau > T} e^{-\int_0^T r_s ds} h (Y_T)} = \EVM{z}{}{e^{-\int_0^T (r_s + \gamma_s) ds}h(Y_T)},
$$
where the equality follows from standard results for stochastic hazard rate models (see e.g.\ \cite[Chapter~10.6]{bib:mcneil-frey-embrechts-15}).
The right hand side can then be computed via Fourier pricing using pricing can then be done via Fourier pricing and evaluation of relation~\eqref{eqn:int_kernerl_rep_payoff}.
In our model  the function $\widetilde{\psi}$ is explicitly available (as in the standard Heston model) as is the transform $\widetilde{h}$ for several popular payoff functions.
The one-dimensional PDE for $\widetilde{\varphi}$ can be efficiently solved with numerical methods such as~\cite{Beckett2001OnEquidistribution}.
Alternatively, the theory of BSDEs provides a probabilistic dual perspective of such PDEs, which is studied in~\cite[Section~4]{ElKaroui1997BackwardFinance}.
More recently, \cite{Han2018SolvingLearning} use such a BSDE approach to design deep learning methods to efficiently solve semilinear PDEs.
Finally, the integral in~\eqref{eqn:int_kernerl_rep_payoff} can be handled with standard Fourier methods as outlined in \cite{carr1999option}, \cite{article} or \cite{articleb}.

\paragraph*{Extensions} Methods for  pricing  credit derivatives where the payment occurs directly at $\tau$ (as in the case of a credit default swap)  with a transform formula such as~\eqref{eqn:extended_transform} are presented in~\cite{Frey2020HowModelsb}.
The above example allows for an immediate extension to a multivariate setting with hazard rates $\gamma^1, \dots, \gamma^{m}$, where for each $i \in \indexI$ the process $\gamma^i$ follows a model of type~\eqref{eqn:creditCIR} (albeit with different parameters and with Brownian motions independent of each other).
Such a specification provides a useful framework for the analysis  of portfolio credit derivatives.
Conditional on $X$, the hazard rates are independent and default dependence amongst the individual firms is generated via their respective loadings on the common factor process $X$.
\cite{Frey2020HowModelsb} use a similar setup with $X$ replaced by a finite-state Markov chain for the pricing and analysis of European safe bonds.

\subsection{Further Applications}

We end this section by sketching two further possible use cases of MMAPs.
The first example deals with Markov-modulated Lévy processes and the second one with Markov-modulated Hawkes processes.

Subordination techniques are a popular approach to build tractable multivariate models for financial assets (see e.g.\ \cite{Luciano2006AModel} or \cite{Luciano2010MultivariateCalibration}).
However, the flexibility of subordinating Lévy processes is limited by the fact that the resulting process is typically again of Lévy type.
Instead, by using MMAPs it is feasible to construct a larger class of processes, even if the underlying modulated process is conditionally of Lévy type.
Additionally, with our class of processes we can precisely target certain aspects of the joint multivariate distribution.
As an example, we consider $Y \in \R^n$ as a model for the log prices of $n$ assets and where we aim at jointly controlling the tail behaviour of the $n$ assets solely via their driving process $X$.
We assume that for each $i \in \{1, \dots, n\}$ the entry $Y^i$ is a pure jump process, and $\varrho_i$ is a random measure selecting the stochastic jump times and sizes of $Y^i$, that is $\sum_{0 < s \leq t} (Y^i_s - Y^i_{s-}) = \int_0^t \int_{\R \setminus \{0\}} \zeta \varrho_i(d\zeta, ds)$, for $t \geq 0$.
The compensator of $\varrho_i(d\zeta, ds)$ is $m_i(X_t, d\zeta)dt$, where $X$ is some exogeneous Markov process satisfying Assumption~\ref{assumption:genX}.
We further assume that
\[
m^i(x, d \zeta) =  \frac{C_i}{|\zeta|^{1 + \Tilde{Y}_i}} \left( \ind{\zeta < 0} e^{-G_i(x)|\zeta|} + \ind{\zeta > 0} e^{-M_i(x)|\zeta|}\right) d \zeta, \quad x \in \stateX,
\]
with constants $C_i > 0, \Tilde{Y}_i < 2$ (note that $C_i$ and $\Tilde{Y}_i$ can of course also depend on $x$, but for the purpose of this example we choose to work with the given simplified setup) and $G_i, M_i \in \Cb{\stateX}$ with $G_i(x), M_i(x) \geq 0$ for each $x$.
For fixed $x$, $m(x, d\zeta)$ is the Lévy kernel of a CGMY process (cf.\ \cite{Carr2002TheInvestigation}).
Symmetry of the return distribution is controlled by the parameters  $G_i(x), M_i(x)$, while for $\Tilde{Y}_i < 0$  the process $Y^i$ has finite activity (otherwise it has infinite activity).
Suppose that the differences $M_i(x) - G_i(x) > 0$ increase with $x$ for each $i \in \{1, \dots, n\}$.
Then, in such a specification the left tails of the returns get heavier as $X$ increases.

Recently, Hawkes processes have received a lot of attention for stock price modelling, both from a statistical perspective and for the purpose of option pricing (see \cite{Bacry2015HawkesFinance}, \cite{Hawkes2018HawkesReview}
\cite{ElEuch2018TheVolatility} or \cite{ElEuch2019TheModels}).
The framework of MMAPs allows the extension of affine Hawkes processes (i.e.\ Hawkes processes with exponential excitation kernel) to the situation where the jump intensities depend on an exogenous Markov process.
For that, consider self-exciting counting processes $(Y^1, \dots, Y^k)$ with values in $\N^k$ with respective intensities $(Y^{k+1}, \dots, Y^m)$, where $m = 2k$.
We further consider a Markov process $X$ which is in line with Assumption~\ref{assumption:genX}.
The intensities are given via
\[
Y^{j + k}_t = Y^{j+k}_0e^{-t \beta_j} + \int_0^t b_{j + k}(X_s) e^{(t - s) \beta_{j}} ds +\sum_{i =1}^k  \int_0^t  \widetilde{\delta}_{ji}e^{(t - s) \beta_{j}} dY^{i}_s, \quad j \in \{1, \dots, k\},
\]
where $b_{k+1}, \dots ,b_m$ are positive bounded (sufficiently regular) functions on $\stateX$, where $(\widetilde{\delta}_{ji})_{i,j = 1}^k \allowbreak = \widetilde\delta$ is a $k\times k$-matrix with positive entries, and where $(\beta_{1}, \dots, \beta_k) \in \R^k$.
We introduce the auxiliary index set $\mathfrak{I} := (k+1, \dots, m)$ and further set $\beta^0 = \operatorname{diag}(\beta_1, \dots \beta_k)$.
Subsequently, the generator of $(X,Y)$ is given by
\begin{align*}
    &\genX f(x, y) + \innerp{(b_{k+1}(x), \dots, b_m(x)) + \beta^0 y_{\mathfrak{I}}}{\nabla_{\mathfrak{I}} f(x, y)} \\
    &\qquad + \sum_{i = 1}^k (f(x, y+ e_i +  \widetilde{\delta}_{1,i}e_{k+1} + \dots + \widetilde{\delta}_{k,i}e_{m}) - f(x, y)) y_{i + k}, \quad f \in \Ckc{2}{\stateXY}.
\end{align*}
We use $\beta$ to denote the $m \times m$-matrix with entries $\beta_{\mathfrak{I}\mathfrak{I}} = \beta^0$ and zeroes otherwise.
Let $\delta_y(d \zeta)$ denote the Dirac measure on $\stateY$ centered at $y$.
Then, we set $\mu_j = \delta_{(e_j - k, \widetilde{\delta}_{\indexI, (j - k)})}$ for $j \in \mathfrak{I}$, to deduce that the $x$-admissible parameters are
\begin{align*}
\big(0, 0, (0_k, b_{k+1}(x), \dots, b_m(x)), \beta, 0, 0, 0, \mu \big),
\end{align*}
where $\mu = (0_k, \mu_{k+1}, \dots, \mu_m)$ and $0_k$ denotes a vector of $k$ zeroes. 

%% file: appendix_defs.tex
\section{Definitions}\label{appendix:defs}
Let $E$ be a locally compact and separable space.
\begin{definition}
A family $(T_t)_{t \geq 0}$ of linear operators on $\Bb{E}$ is a \emph{semigroup}, if
\[
T_0 = \Id, \quad T_t T_s f = T_s T_t f = T_{t + s}f\quad \forall f \in \Bb{E}, \, s, t \geq 0.
\]
A semigroup is \emph{sub-Markov}, if for all $t \geq 0$
\begin{align}
    \label{eqn:def_positivity}
T_t f &\geq 0, \quad \forall f \in \Bb{E} \text{ with } f \geq 0, \\
    \label{eqn:sub_Markov_property}
T_t f &\leq 1, \quad \forall f \in \Bb{E} \text{ with } f \leq 1.
\end{align}
\end{definition}
Property~\eqref{eqn:def_positivity} is typically referred to as positivity preserving and \eqref{eqn:sub_Markov_property} is the sub-Markov property.

\begin{definition}\label{def:feller}
A \emph{Feller semigroup} is a sub-Markov semigroup satisfying the Feller property
\[
T_t f \in \Cinf{E} \quad \forall f \in \Cinf{E}, \, t > 0,
\]
and which is strongly continuous in $\Cinf{E}$, i.e.
\[
\lim_{t \to 0} \norm{T_t f - f}_\infty = 0 \quad \forall f \in \Cinf{E}.
\]
\end{definition}

Note that in the terminology of~\cite{bib:ethier-kurtz-86} a Feller semigroup on $\Cinf{E}$ is a strongly continuous, positivity preserving and conservative contraction semigroup.
The above definition~\ref{def:feller} is in line with the one given in e.g.\ \cite{revuz1999continuous}.
The following definition introduces a concept weaker than Feller semigroup. 
For its definition, we equip $\Cb{E}$ with locally uniform convergence.

\begin{definition}\label{def:Cb_feller}
A \emph{$C_b$-Feller semigroup} is a sub-Markov semigroup satisfying
\[
T_t f \in \Cb{E} \quad \forall f \in \Cb{E}, \, t > 0,
\]
and for which $t \mapsto T_t f$ is continuous in the topology of locally uniform convergence in $\Cb{E}$.
\end{definition} 

%% file: appendix_proofs.tex
\section{Additional Proofs and results }\label{appendix:proof}

\begin{proof}[Proof of Proposition~\ref{cor:afterMPresult}.]
We begin with the situation, where $f$ and $g$ are independent of the time variable $t \in [0, \infty)$, that is we fix arbitrary $f \in \domgen{\weakgenX}$, $g \in \Ckb{2}{\stateY}$ and introduce the process
\[
M^{f,g}_t := f(X_t)g(Y_t) - f(x)g(y) - \int_0^t (\weakgenX + \genY) f(X_s) g(Y_s) ds, \quad t \geq 0.
\]
Note that by assumption $\weakgenX f \in \Cb{\stateX}$.
We pick a sequence $\{f_k\} \subset \Ckc{2}{\stateX}$ with $\bplim{k \to \infty} \allowbreak f_k = f$.
By the Feller property of $(\CsemiX_t)$ we have that $\genX f_k \in \Cinf{\stateX}$ for each $k$.
Since $\Cb{\stateX}$ is dense in $\Bb{\stateX}$ with respect to bp-convergence (cf.\ \cite[Proposition~3.4.2]{bib:ethier-kurtz-86}), we know that there exists some $\Tilde{f} \in \Bb{\stateX}$ with $\bplim{k \to \infty} \genX f_k = \Tilde{f}$.
Choose arbitrary $t > 0$ and $x \in \stateX$ and consider
\[
\frac{\CsemiX_t f_k(x) - f_k(x)}{t} = \frac{1}{t} \int_0^t \CsemiX_s \genX f_k(x) ds.
\]
Take for the above equation the limit $k \to \infty$ and note that $\CbsemiX_t f \in \domgen{\weakgenX}$ (see \cite[Lemma~2.2.3]{lorenzi2006analytical}) to arrive at
\begin{align*}
    \frac{1}{t} \int_0^t \CbsemiX_s \weakgenX f(x) ds &= \frac{\CbsemiX f(x) - f(x)}{t}= \frac{1}{t} \int_0^t \CbsemiX \Tilde{f}(x) ds,
\end{align*}
which shows that $\bplim{k \to \infty} \genX f_k = \weakgenX f$.
Moreover, by Lemma~\ref{lemma:bpconv_genY} below we can pick a sequence $(g_k) \in \Ckc{2}{\stateY}$ with $\bplim{k \to \infty} (g_k, \genY g_k) = (g, \genY g)$.

Overall, $\bplim{k \to \infty}(f_k g_k, \gen f_k g_k) = (fg, (\weakgenX + \genY) fg)$, and by dominated convergence (and the associated $L^1$-convergence), we conclude that $M^{f,g}$ is a martingale.
Getting back to the initial (time-dependent) functions $f, g$, we fix arbitrary $t_2 > t_1 \geq 0$, and we consequently get
\begin{align*}
    &\EVM{z}{}{f(t_1, X_{t_2})g(t_1, Y_{t_2}) - f(t_1, X_{t_1})g(t_1, Y_{t_1}) \mid \calF_{t_1}} \\
    &\qquad \qquad \qquad \qquad = \EVM{z}{}{\int_{t_1}^{t_2} (\weakgenX + \genY) f(t_1, X_s)g(t_1, Y_s) ds \mid \calF_{t_1}}
\end{align*}
Moreover, by the fundamental theorem of calculus, we have
\begin{align*}
    &\EVM{z}{}{f(t_2, X_{t_2})g(t_2, Y_{t_2}) - f(t_1, X_{t_2})g(t_1, Y_{t_2}) \mid \calF_{t_1}} \\
    &\qquad \qquad \qquad \qquad = \EVM{z}{}{\int_{t_1}^{t_2} \partial_t \big( f(s, X_{t_2})g(s, Y_{t_2}) \big) ds \mid \calF_{t_1}},
\end{align*}
and so by Lemma~4.3.4 of \cite{bib:ethier-kurtz-86} we arrive at the desired result (set $v = \partial_t fg$ and $w = (\weakgenX + \genY) fg$ to be in line with the notation of that lemma).
\end{proof}

\begin{proof}[Proof of Lemma~\ref{lemma:boundednessF}.]
To show the boundedness of $F(\cdot, u)$, we clearly only need to deal with the integral part involving $m(\cdot, d\zeta)$.
We start with a useful Taylor expansion,
\begin{equation}\label{eqn:boundedFTaylor}
\begin{aligned}
    &\left(e^{\innerp{u}{\zeta}} - e^{\innerp{u_\indexJ}{\zeta_\indexJ}}\right) + \left(e^{\innerp{u_\indexJ}{\zeta_\indexJ}} - 1 - \innerp{u_\indexJ}{\zeta_\indexJ}\right) \\
    &\qquad = \sum_{i \in \indexI } u_i \zeta_i e^{\innerp{u_\indexJ}{\zeta_\indexJ}}  \int_0^1 e^{\innerp{u_\indexI}{s\zeta_\indexI}} ds + \sum_{i,j \in \indexJ} u_i u_j \zeta_i \zeta_j  \int_0^1 e^{\innerp{u_\indexJ}{s\zeta_\indexJ}} (1 - s) ds.
    \end{aligned}
\end{equation}
Introduce the sets $Q_0 := \{\zeta \in \stateY \mid |\zeta_i| \leq 1, \, 1 \leq i \leq n\}$ and $Q_0^0 := Q_0\backslash \{0\}$, and note that $\chi_\indexJ(\zeta) = \zeta_\indexJ$ on $Q_0$.
We use~\eqref{eqn:boundedFTaylor} and compute
\begin{align*}
    &\int_{\stateY \backslash \{0\}} |e^{\innerp{u}{\zeta}} - 1 - \innerp{u_\indexJ}{\chi_\indexJ(\zeta)}|m(x, d\zeta) \\
    &\qquad \leq C_1 \int_{Q_0^0} |\innerp{\chi_\indexI(\xi)}{\bm{1} } + \norm{\chi_\indexJ(\xi)}^2| m(x, d\xi) \\
    &\qquad \qquad +  \int_{\stateY \backslash Q_0} |e^{\innerp{u}{\zeta}} + 1 - \innerp{u_\indexJ}{\chi_\indexJ(\zeta)}|m(x, d\zeta) \\
    &\qquad \leq C_1 M(x, Q_0^0) + C_2 M(x, \stateY \backslash Q_0),
\end{align*}
for some constants $C_1, C_2$.
Clearly, these estimates hold locally uniformly in $u$.
The assertion then follows since $\sup_{x \in \stateX}M(x, \stateY \backslash \{0\}) < \infty$, cf.\ Definition~\ref{def:xadmissible}.
\end{proof}

\begin{lemma}\label{lemma:bpconv_genY}
Let $x \in \stateX$ and $g \in \Ckb{2}{\stateY}$ be arbitrary.
Then there exists a sequence $(g_k)_{k \in \N}$ with elements in $\Ckc{2}{\stateY}$ such that $\bplim{k \to \infty} (g_k, \genY g_k) = (g, \genY g)$.
\end{lemma}

\begin{proof}
As in the proof of~\cite[Proposition 8.2]{Duffie2003AffineFinance}, we choose a function $\rho \in \Ckc{\infty}{\R_+}$ with
\[
\rho(r) = \begin{cases} 1, \quad & \text{if } r \leq 1, \\
0, \quad & \text{if } r > 5,\end{cases}
\]
as well as $0 \leq \rho(r) \leq 1$, $|\partial_r \rho(r)| \leq 1$ and $|\partial_r^2 \rho(r)| \leq 1$ for all $r \in \R_+$.
We introduce a sequence of functions $g_k \in \Ckc{2}{\stateY}, \, k \in \N$, via
\begin{equation}
    \label{eqn:approxCharAux1}
    g_k(y) := g(y) \rho(\norm{y}^2 / k).
\end{equation}
It is easy to see that $\bplim{k \to \infty} g_k = g$.
Even more so, for arbitrary $(x, y) \in \stateXY$ it holds that
\begin{align*}
    &|\genY \left(g_k - g\right)(y)| \\
    &\qquad \leq  \sum_{i,l=1}^n \norm{a_{il}}_\infty \Big\{|\partial_{y_i} \partial_{y_l} g(y)|\big(1 - \rho(\norm{y}^2 / k) \big)  \\
    &\qquad \quad + \frac{2 |y_i + y_l|}{k} \rho'(\norm{y}^2 / k) \norm{(\partial_{y_i} + \partial_{y_l})g}_\infty + \frac{4 |y_l y_i |}{k^2} \rho''(\norm{y}^2 / k) \norm{g}_\infty \Big\} \\
    &\qquad \quad + \frac{2}{k} \norm{g}_\infty \rho'(\norm{y}^2 / k) \sum_{i = 1}^n \norm{a_{ii}}_\infty \notag \\
    &\qquad \quad + \sum_{i,l=1}^n |\innerp{\alpha_{\indexI,il}}{y^+}| \Big\{|\partial_{y_i} \partial_{y_l} g(y)|\big(1 - \rho(\norm{y}^2 / k) \big) \notag\\
    & \qquad \quad + \frac{2 |y_i + y_l|}{k} \rho'(\norm{y}^2 / k) \norm{(\partial_{y_i} + \partial_{y_l})g}_\infty + \frac{4 |y_l y_i |}{k^2} \rho''(\norm{y}^2 / k) \norm{g}_\infty \Big\} \notag\\
    &\qquad \quad + \frac{2}{k} \norm{g}_\infty \rho'(\norm{y}^2 / k) \sum_{i = 1}^n |\innerp{\alpha_{\indexI,il}}{y^+}| \notag\\
    & \qquad \quad +\sum_{i = 1}^n \norm{b_i}_\infty \Big\{|\partial_{y_i} g(y)|\big(1 - \rho(\norm{y}^2 / k) \big) + \frac{2|y_i|}{k} \rho'(\norm{y}^2 / k)\norm{g}_\infty \Big\} \notag \\
    & \qquad \quad +\sum_{i = 1}^n |(\beta y)_i| \Big\{|\partial_{y_i} g(y)|\big(1 - \rho(\norm{y}^2 / k) \big) + \frac{2|y_i|}{k} \rho'(\norm{y}^2 / k)\norm{g}_\infty \Big\} \notag \\
    & \qquad \quad + \norm{c}_\infty |g(y) - g_k(y)| + \innerp{\gamma}{y^+} |g(y) - g_k(y)| \notag \\
    & \qquad \quad + \int_{\stateY\setminus \{0\}} \big| g(y + \zeta)(1 - \rho(\norm{y + \zeta}^2 / k) - g(y)(1 - \rho(\norm{y}^2 / k) \notag \\
    & \qquad \quad-  (1 - \rho(\norm{y}^2 / k)) \innerp{\nabla_\indexJ g(y)}{\chi_\indexJ(\zeta)}\big| m(x, d \zeta) \\
    & \qquad\quad + \sum_{j = m + 1}^{n} \frac{2|y_j|}{k} \rho'(\norm{y}^2 / k) \norm{g}_\infty  \norm{M}_\infty \notag \\
    & \qquad \quad + \sum_{i = 1}^m \int_{\stateY\setminus \{0\}} \big| g(y + \zeta)(1 - \rho(\norm{y + \zeta}^2 / k) - g(y)(1 - \rho(\norm{y}^2 / k) \notag \\
    & \qquad \quad-  (1 - \rho(\norm{y}^2 / k)) \innerp{\nabla_{\indexJ(i)} g(y)}{\chi_\indexJ(\zeta)}\big| y_i^+ \mu_i( d \zeta) \\
    & \qquad \quad + \sum_{i = 1}^{m} \frac{2\norm{y_\indexJ}}{k} \rho'(\norm{y}^2 / k) \norm{g}_\infty y_i^+\mathcal{M}_i, \notag
\end{align*}
which converges to $0$ as $k \to \infty$ and so we get $\bplim{k \to \infty} \genY g_k = \genY g$.
\end{proof}